\documentclass[11pt]{article}
\usepackage{color}
\usepackage[utf8]{inputenc}
\usepackage{amsmath,amsthm,amssymb,amscd}
\usepackage[all,cmtip]{xy}
\usepackage{booktabs} 
\usepackage[hyperfootnotes=false, colorlinks, linkcolor={blue}, citecolor={magenta}, filecolor={blue}, urlcolor={blue}, plainpages=false, pdfpagelabels]{hyperref}
\usepackage[overload]{textcase}
\usepackage{tikz}
\usetikzlibrary{decorations.pathmorphing}
\usepackage[export]{adjustbox}
\usepackage{tikz-cd}
\usepackage{mathtools}
\usepackage[numbers,sort&compress]{natbib} 
\usepackage{geometry}
\usepackage{multirow}
\allowdisplaybreaks
\newcommand{\A}{{\mathbb A}}
\newcommand{\Q}{{\mathbb Q}}
\newcommand{\Z}{{\mathbb Z}}

\newcommand{\C}{{\mathbb C}}

\newcommand{\p}{\mathfrak p}
\newcommand{\OF}{{\mathfrak o}}
\newcommand{\GL}{{\rm GL}}

\newcommand{\SO}{{\rm SO}}
\newcommand{\Sp}{{\rm Sp}}
\newcommand{\GSp}{{\rm GSp}}
\newcommand{\PGSp}{{\rm PGSp}}

\newcommand{\St}{{\rm St}}

\newcommand{\K}[1]{{\rm K}(\p^{#1})}

\newcommand{\ind}{{\rm ind}}

\newcommand{\sym}{{\rm sym}}
\newcommand{\forget}[1]{}

\def\qdots{\mathinner{\mkern1mu\raise0pt\vbox{\kern7pt\hbox{.}}\mkern2mu
\raise3.4pt\hbox{.}\mkern2mu\raise7pt\hbox{.}\mkern1mu}}

\newcommand\blfootnote[1]{%
	\begingroup
	\renewcommand\thefootnote{}\footnote{#1}%
	\addtocounter{footnote}{-1}%
	\endgroup
}

\newtheorem{lemma}{Lemma}[subsection]
\newtheorem{theorem}[lemma]{Theorem}
\newtheorem{corollary}[lemma]{Corollary}

\newtheorem{remark}[lemma]{Remark}

\newtheorem{theorem1}{Theorem}

\title{Paramodular forms coming from elliptic curves}
\author{Manami Roy}
\date{}

\begin{document}
	
\maketitle
\blfootnote{\hspace*{0.2in } Key words and phrases. Elliptic curves; Local representations; Paramodular forms; Symmetric cube.}

\begin{abstract}
\noindent There is a lifting from a non-CM elliptic curve $E/\Q$ to a paramodular form $f$ of degree $2$ and weight $3$ given by the symmetric cube map. We find the level of $f$ in terms of the coefficients of the Weierstrass equation of $E$. In order to compute the paramodular level, we use the available description of the local representations of $\GL(2,\Q_p)$ attached to $E$ for $p \geq 5$ and determine the local representation of $\GL(2,\Q_3)$ attached to $E$.
\end{abstract}

\tableofcontents

\section*{Introduction}
In this paper, we study certain Siegel cusp forms of degree $2$ with arithmetic significance. In 2007, D.\ Ramakrishnan and F.\ Shahidi \cite{RamakrishnanShahidi2007} proved a transfer from a non-CM elliptic curve $E$ over $\Q$ to a Siegel cusp form $f$ of degree $2$ and weight $3$ such that the spin $L$-function of $f$ is equal to the symmetric cube $L$-function of $E$. We call this transfer the symmetric cube transfer (or the $\sym^3$ lifting); see Section~\ref{Section 1} for details. The goal of this paper is to better understand the Siegel cusp forms coming from this transfer using some recent results which were not available at that time. In order to study Siegel modular forms, it is important to specify a congruence subgroup. Some natural questions are: Which congruence subgroup should we consider to study the Siegel cusp forms coming from elliptic curves? What is the level of a Siegel cusp form obtained by this transfer with respect to a specific congruence subgroup? In this article, we consider Siegel cusp forms with respect to the paramodular group of level $M$,
\begin{equation}
\label{paramodular group}
K(M)= \Sp(4,\Q) \cap \begin{bsmallmatrix} \Z&M\Z&\Z&\Z\\ \Z&\Z&\Z&M^{-1}\Z\\ \Z&M\Z&\Z&\Z\\ M\Z&M\Z&M\Z&\Z\end{bsmallmatrix}.
\end{equation}
A Siegel modular form with respect to the paramodular group of some level is called a \emph{paramodular form}. The following result, which is a consequence of Corollary~\ref{global corollary2}, describes the level of the paramodular form coming from a non-CM elliptic curve over $\Q$ via the $\sym^3$ lifting. We write $\Gamma_{\C}(s)=2(2\pi)^{-s}\Gamma(s)$ with $\Gamma$ being the usual gamma function.
\begin{theorem1}
	\label{theorem in intro}
	Let $E$ be a non-CM elliptic curve over $\Q$ given by the global minimal Weierstrass equation \eqref{W.E of EC} with coefficients in $\Z$ and conductor $N$. Let $\Delta$ be the discriminant attached to the given Weierstrass equation and $v_p$ be the $p$-adic valuation.  Let $\Delta'=3^{-v_3(\Delta)}\Delta$. Suppose that $E$ has good or multiplicative reduction at $p=2$. Then there is a cuspidal paramodular newform $f$ of degree $2$, weight $3$ and level $M$ with the following properties:
	\begin{enumerate}
		\item[(i)]  The level $M$ of $f$ is given by ${\displaystyle M= N\prod_{\substack{p \mid N \\ v_p(\Delta) \not\equiv 0\ \mathrm{mod}\ 4}}p^2}$.
		\item[(ii)]  
		The completed spin $L$-function $L(s,f)$ attached to $f$ is given by 
		\begin{equation*}
		L(s,f)=\Gamma_{\C}\left(s+\frac 32\right)\Gamma_{\C}\left(s+\frac 12\right)\prod_{p<\infty} L_p(s,f),
		\end{equation*}
		where $ L_p(s,f)=L_p(s,E,\sym^3)$ for all places $p$. If $p|N$, then 
		\begin{equation*}
		L_p(s,f)=
		\begin{cases}
		\frac{1}{1-p^{- 3/2-s}}&\text{ if }E \text{ has split multiplicative reduction at } p,\\
		\frac{1}{1+p^{-3/2-s}}&\text{ if }E \text{ has non-split multiplicative reduction at } p,\\
	\frac{1}{(1-\alpha p^{-s})(1-\alpha^{-1} p^{-s})} &\text{ if }   j(E) \in \Z_p\text{ and }v_p(\Delta) \equiv 0\ \mathrm{mod}\ 4,\\
		1 &\text{ otherwise}.
		\end{cases}
		\end{equation*}
			Here, $\alpha$ is an element of $\C^{\times}$ such that $|\alpha|=1$. If the following condition is satisfied
		\begin{equation}
		\label{condition for Euler factor in theorem 1}
	 j(E) \in \Z_p,\ v_p(\Delta) \equiv 0\ \mathrm{mod}\ 4,	\text{ and }
		\left\{\begin{array}{ll}
	    (p-1)v_p(\Delta) \not\equiv 0\ \mathrm{mod}\ 12 &\text{if } p\ge5,\\[1ex]
		\left(\frac {\Delta'}{3}\right)=-1&\text{if } p=3,
		\end{array}\right\}
		\end{equation}
		then $\alpha=i$ (the fourth root of unity).
		\item[(iii)] The Atkin-Lehner eigenvalues of $f$ at the finite places are given by 
		\begin{equation*}
		\eta_p=
		\begin{cases}
		-1 &\text{ if }  p\mid N, \text{ and } E \text{ has split mult.\ red.\ at } p \text{ or satisfies } \eqref{condition for Atkin-Lehener in theorem 1},\\
	w(E/\Q_3)&\text{ if } 3\mid N,\ p=3 \text{ and } E \text{ satisfies } {\rm S}_6^{'} \text{ or } {\rm S}_6^{''},\\
	1 & \text{ otherwise},
		\end{cases}
		\end{equation*}
			where ${\rm S}_6^{'}$, ${\rm S}_6^{''}$ are defined in Table~\ref{IndexNeron} and  the condition \eqref{condition for Atkin-Lehener in theorem 1} is given by \begin{equation}
			\label{condition for Atkin-Lehener in theorem 1}
			j(E) \in \Z_p, \text{ and }
			\left\{\begin{array}{ll}
		v_p(\Delta) \equiv 0\ \mathrm{mod}\ 4, (p-1)v_p(\Delta) \not\equiv 0\ \mathrm{mod}\ 12 &\text{if } p\ge5,\\[1ex]
		v_p(\Delta) \equiv 2\ \mathrm{mod}\ 4, \left(\frac {\Delta'}{3}\right)=-1&\text{if } p=3,
			\end{array}\right\}.
			\end{equation}
	\end{enumerate}
\end{theorem1}
The quantity $w(E/\Q_3)$  appearing in part (iii) of the theorem is the local root number of $E/\Q_3$, one can compute this quantity using Theorem~1.1 of \cite{Kobayashi2002}. The $L$-function $L(s,f)$ in the theorem satisfies the functional equation $L(s,f)=\varepsilon(s,f)L(1-s,f)$, where $\varepsilon(s,f)=-\left(\prod_{p<\infty} \eta_p\right)N^{1/2-s}$.  For example, the elliptic curve with the global minimal Weierstrass equation $y^2+xy+y=x^3-x-2$ has discriminant $\Delta=-2\cdot5^4$. So, by Theorem~\ref{theorem in intro}, there exists a paramodular form $f$ of degree $2$, weight $3$ and level $200$. Moreover, 
$L_5(s,f)=(1+5^{-2s})^{-1}$, $L_2(s,f)=(1+2^{-3/2-s})^{-1}$, $\eta_5=-1$, and $\eta_2=1$. D.\ Yuen and C.\ Poor have some heuristic data on $S_3(K(200))$ and $f$ in the above example appears as one of the non-Gritsenko lift eigenforms in their data.

In order to prove this result, we need to understand the underlying representation theoretic mechanism. Note that attached to every non-CM elliptic curve, there is a cuspidal automorphic representation of $\GL(2,\A_{\Q})$. Similarly, Siegel cusp forms (especially cuspidal paramodular forms) of degree $2$ are related to cuspidal automorphic representations of $\GSp(4,\A_{\Q})$. Theorem~\ref{RS} describes the symmetric cube transfer of cuspidal automorphic representations from $\GL(2,\A_{\Q})$ to $\GSp(4,\A_{\Q})$. We give a new proof of Theorem~\ref{RS} using some recent techniques. Then, we discuss the level of the Siegel modular forms coming from non-CM newforms via the $\sym^3$ transfer. A general description for the level of such paramodular forms is given in Section~\ref{Section 1} using Corollary~\ref{RS-classical} and Table~\ref{gerenal conductor of sym3}. To find a complete description of the paramodular form $f$ as in Theorem~\ref{theorem in intro}, we consider local representations attached to $f$; in this sense our work is similar to the work of B.\ Roberts and J.\ Johnson-Leung \cite{JohnsonRoberts2012} who lift Hilbert modular forms to paramodular forms. 

We concentrate on elliptic curves over $\Q$ to get a simple formula for the level of some paramodular forms coming from the $\sym^3$ lifting. To calculate the paramodular level attached to an elliptic curve $E/\Q$, it is necessary to know the local representations of $\GL(2,\Q_p)$ attached to $E$. Finding a definite characterization of the local representations of $\GL(2,K)$ associated to elliptic curves over a non-archimedean local field $K$ of characteristic zero is of independent interest; we discuss this in Section~\ref{Section 2}. Particularly, this article treats the case of residual characteristic $3$, i.e., we find the local representations of $\GL(2,K)$ attached to elliptic curves with additive but potentially good reduction over a non-archimedean local field $K$ of characteristic zero and residual characteristic $3$. 

Finally, using the results from Section~\ref{Section 1} and \ref{Section 2}, we find the level of the paramodular forms coming from non-CM elliptic curves over $\Q$ via the $\sym^3$ lifting in Section~\ref{Section 3}. Dummigan, Martin, and Watkins \cite{DummiganMartinWatkins2009} computed the Euler factors and local Atkin-Lehner signs for the symmetric $n$th-power $L$-function of $E$ by analyzing the Galois representations at each prime $p$. We have reproduced these local quantities for the $\sym^3$ case in Section~\ref{sym3 of local reprepesentations}.
\section{Siegel modular forms obtained by the \texorpdfstring{${\rm sym^3}$}{} lifting}
\label{Section 1}
There is an action of $\GL(2,\C)$ on the space of homogeneous polynomials of degree $3$ in $\C[S,T]$ given by
\begin{equation*}
\begin{bsmallmatrix}a &b\\c&d\end{bsmallmatrix}\cdot f(S,T)= f(aS+cT,bS+dT).
\end{equation*}
This action defines a four-dimensional irreducible representation of $\GL(2,\C)$ and induces a map from $\GL(2,\C)$ to $\GL(4,\C)$. After conjugating by a suitable element of $\GL(4,\C)$, we may assume that the image of the map lies in $\GSp(4,\C)$. This map is known as the symmetric cube map, denoted as {${\rm sym^3}$}, and given by

\begin{align}
\label{sym^3map}
\sym^3: \GL(2,\C) &\rightarrow  \GSp(4,\C) \nonumber \\
\begin{bmatrix}a &b\\c&d\end{bmatrix} &\mapsto\begin{bmatrix}a^3 &a^2b&ab^2&-\frac 13 b^3\\3a^2c&2abc+a^2d&2abd+b^2c&-b^2d\\
3ac^2&2acd+bc^2&2bcd+ad^2&-bd^2\\-3c^3&-3c^2d&-3cd^2&d^3\end{bmatrix}.
\end{align}
Here, $\GSp(4)$ is the group of symplectic similitudes, defined by
\begin{equation*}
\GSp(4)=\left\{g \in \GL(4):{}^tgJg=\lambda(g)J \text{ for some } \lambda(g) \in \GL(1)  \right\} ,
\end{equation*}
where $J$ is the symplectic form given by the matrix $J=\begin{bsmallmatrix} &&&1\\&&1&\\&-1&&\\-1&&&\end{bsmallmatrix}$. We choose this particular symplectic form for simplicity of calculations over the local fields.

\subsection{Representation theoretic aspect of the \texorpdfstring{${\rm sym^3}$}{} lifting}
Here, a ``lifting" means a functorial lift according to {\it the Langlands principle of functoriality}, a central conjecture in the Langlands program that describes the relationships between automorphic objects living on two different algebraic groups. Let $G, H$ be two (split) reductive algebraic groups defined over $\Q$. Attached to the groups $G$ and $H$ are their dual groups $\hat{G}$ and $\hat{H}$, which are complex reductive Lie groups whose root systems are dual to those of $G$ and $H$ respectively. By the principle of functoriality, every homomorphism of Lie groups $\hat{G} \rightarrow \hat{H}$ should give rise to a ``lifting" of automorphic representations of $G(\A_{\Q})$ to automorphic representations of $H(\A_{\Q})$ that connects their $L$-functions. Now, we have the map $\sym^3: \widehat{\GL(2)}=\GL(2,\C) \rightarrow \widehat{\GSp(4)}=\GSp(4,\C)$. So the principle of functoriality predicts that an automorphic representation of $\GL(2,\A_{\Q})$ should ``lift" to an automorphic representation of $\GSp(4,\A_{\Q})$, i.e., the following diagram should hold:
\begin{equation}
\label{Functoriality for sym3}
\left\{ \begin{tikzcd}\GL(2,\C) \arrow[r, "\sym^3"]& \GSp(4,\C) \end{tikzcd}\right\} \textcolor{blue}{\xRightarrow[\text{Functoriality}]{\text{Langlands}}} \Bigg\{ \begin{tikzcd}\\[-0.33 in] \text{auto.\ rep.\ of} &\text{auto.\ rep.\ of}\\[-0.33 in] \GL(2,\A_{\Q}) \arrow{r}{\sym^3}[swap]{\text{lifting}}&\GSp(4,\A_{\Q})\end{tikzcd}\Bigg\}.
\end{equation}
Ramakrishnan and Shahidi \cite{RamakrishnanShahidi2007} proved the following lifting from a cuspidal automorphic representation $\GL(2,\A_{\Q})$ to a cuspidal automorphic representation $\GSp(4,\A_{\Q})$.
\begin{theorem}[Ramakrishnan-Shahidi, 2007]
	\label{RS}
	Let $\pi=\bigotimes_{p}\pi_p$ be a cuspidal automorphic representation of $\GL(2,\A_{\Q})$ defined
	by a holomorphic, non-CM newform $\phi$ of even weight $k \geq 2$ and level N with trivial central character. Then there exists a cuspidal automorphic representation $\Pi=\bigotimes_{p}\Pi_p$ of $\GSp(4,\A_{\Q})$ with trivial central character, which is unramified at any prime 
	p not dividing N, such that
	\begin{enumerate}
		\item[(i)] Each non-archimedean component $\Pi_p$ is generic, with its parameter being $\sym^3$ of the  parameter of $\pi_p$.
		\item[(ii)] $\Pi_{\infty}$ is a holomorphic discrete series representation, with its parameter being $\sym^3$ of the archimedean parameter of $\pi$.
		\item[(iii)] $L(s,\Pi)=L(s, \pi, \sym^3)$.
	\end{enumerate}
\end{theorem}
\noindent Using some results of James Arthur \cite{Arthur2013} and Ralf Schmidt \cite{Schmidt2018} that were not available at the time, we can now shorten the proof of Theorem~\ref{RS} as follows:
\begin{proof}
By the functorial symmetric cube transfer for cuspidal automorphic representations of $\GL(2,\A_{\Q})$ in \cite{KimShahidi2002}, we get a unitary cuspidal automorphic representation $\mu=\otimes_p \mu_p$ of $\GL(4,\A_{\Q})$ with trivial central character from $\pi$ such that $ L(s,\mu)=L(s, \pi, \sym^3)$. Also, $\mu$ has the following properties:
\begin{itemize}
	\item $\mu$ is symplectic, i.e., the exterior square $L$-function $L(s,\mu,\Lambda^2)$ has a pole at $s=1$. This is true because of the well known identity $L(s,\mu,\Lambda^2)=L(s,\pi,\sym^4)\zeta(s)$ and the fact that $L(s,\pi,\sym^4)$ has no zero at $s=1$. 
	\item $\mu$ is self-dual. This follows from the identity $L(s,\mu\times\mu)=L(s,\mu,\Lambda^2)L(s,\mu,\sym^2)$ and the facts that $L(s,\mu,\sym^2)$ has no zero at $s=1$ and the Rankin-Selberg $L$-function has a pole at $s=1$ if and only if $\mu$ is isomorphic to its contragredient $\mu^{\vee}$.
\end{itemize}
These two properties are explained in detail in \cite{RamakrishnanShahidi2007}. Now, since $\mu$ is a self-dual, symplectic, unitary, cuspidal automorphic representation of $\GL(4,\A_{\Q})$, $\psi=\mu \boxtimes 1$ is an Arthur parameter of general type for the group $\SO(5)$. Then by Arthur's classification in \cite{Arthur2013}, there is a packet $\Pi_\psi$ of cuspidal automorphic representations of $\SO(5,\A_{\Q})$ with $\psi$ as the Arthur parameter. Since $\SO(5)$ and $\PGSp(4)$ are isomorphic as algebraic groups, a representation of $\SO(5,\A_{\Q})$ can be viewed as a representation of $\GSp(4,\A_{\Q})$ with trivial central character. Furthermore, using Theorem~1.1 of \cite{Schmidt2018}, it follows that there exists an element $\Pi\cong\otimes\Pi_p$ in the packet $\Pi_\psi$ such that each non-archimedean local representation $\Pi_p$ is a generic representation and $\Pi_{\infty}$ is a holomorphic discrete series representation. We choose such a representation $\Pi$ in $\Pi_\psi$. So, we get a cuspidal automorphic representation $\Pi$ of $\GSp(4,\A_{\Q})$ with trivial central character which satisfies properties (i) and (ii) of Theorem~\ref{RS}. Also, it is evident from Lemma~1.2 of \cite{Schmidt2018} that the spin $L$-function of $\Pi$ coincides with the standard $L$-function $ L(s,\mu)$, i.e., $L(s,\Pi)=L(s,\mu)=L(s, \pi, \sym^3)$. 
\end{proof}
Moreover, this $\sym^3$ lifting (\ref{Functoriality for sym3}) is functorial at each place since the local Langlands correspondence is true for $\GL(2)$ and $\GSp(4)$. So, we can specifically study the ``local components" of the $\sym^3$ lifting, i.e., for an irreducible and admissible representation $\pi_p$ of $\GL(2,\Q_p)$, we can consider the representations $\Pi_p$ of $\GSp(4,\A_{\Q_p})$ whose $L$-parameter is the same as the $L$-parameter of $\mu_p$. For the purposes of this paper, it is important to study the ``local $\sym^3$ lift" at each place. 
\subsection{Level of the Siegel modular forms coming from the \texorpdfstring{${\rm sym^3}$}{} lifting}
\noindent For a general reference on Siegel modular forms in this article, we use \cite{Schmidt2017}. We are interested in a classical version of Theorem~\ref{RS} which associates a Siegel modular form to a non-CM cuspidal newform via the $\sym^3$ lifting. Moreover, we want to understand the level of the Siegel modular forms obtained by this lifting.  Ramakrishnan and Shahidi considered the principal congruence subgroup level, but principal congruence subgroups are not ideally suited for a correspondence between cuspidal automorphic representations of $\GSp(4,\A_{\Q})$ and Siegel modular forms. 
Instead we consider Siegel modular forms with respect to the paramodular group as in \eqref{paramodular group}.
There is a well understood connection between paramodular forms and cuspidal automorphic representations of $\GSp(4,\A_{\Q})$, and there is a nice newform theory for paramodular forms (see \cite{RobertsSchmidt2006,RobertsSchmidt2007}). These facts were not available at the time Ramakrishnan and Shahidi proved the result on the $\sym^3$ lifting. Now, using Theorem \ref{RS} and the paramodular newform theory we get the following corollary:
\begin{corollary}
\label{RS-classical}
Let $\phi$ be a non-CM cuspidal newform of even weight $k \geq 2$ and level N with trivial central character. Let $\pi=\bigotimes_{p}\pi_p$ be the cuspidal automorphic representation of $\GL(2,\A_{\Q})$ associated to $\phi$. Then there exists a vector valued cuspidal paramodular newform $f$ of weight ${\rm det}^{k+1}\times \sym^{k-2}$ and level equal to the conductor $a(\sym^3(\pi))$ of $\sym^3(\pi)$ such that $L(s,f)=L(s, \phi, \sym^3)$.
\end{corollary}
\begin{proof}
By Theorem~\ref{RS}, $\pi$ lifts to a cuspidal automorphic representation $\Pi=\sym^3(\pi)$ of $\GSp(4,\A_{\Q})$ with trivial central character. Moreover, $\Pi$ is a representation of type {\bf(G)} (see  \cite{Schmidt2018}) such that each non-archimedean component $\Pi_p$ is generic and the archimedean component $\Pi_{\infty}$ is a holomorphic discrete series representation. Then, by Theorem~4.4.1 from \cite{RobertsSchmidt2007}, there exists a paramodular vector of minimal level at each non-archimedean place. Now, recall from Theorem~7.5.4 and Corollary~7.5.5 of \cite{RobertsSchmidt2007} that the minimal paramodular level of $\Pi_p$ at each finite place $p$ is the conductor $a(\Pi_p)$ of $\Pi_p$, and the dimension of the space of paramodular vectors at the minimal level is $1$. Since $a(\Pi)=\prod_{p} p^{a(\Pi_p)}$, we get a paramodular newform $f$ of level $a(\Pi)=a(\sym^3(\pi))$ associated to the automorphic representation $\Pi$ of $\GSp(4,\A_{\Q})$ such that $L(s,f)=L(s,\Pi)$ (see \cite{Schmidt2017}). Using part (ii) of Theorem~\ref{RS} we get $L(s,f)=L(s,\Pi)=L(s, \pi, \sym^3)=L(s, \phi, \sym^3)$.
Also, one can easily see that the weight of $f$ is ${\rm det}^{k+1}\times \sym^{k-2}$ by looking at the local parameter at the archimedean place and using the map (\ref{sym^3map}) (see \cite{Schmidt2017}). This concludes the proof.
\end{proof}
\begin{remark}
In order to construct holomorphic Siegel modular forms, we choose the non-archimedean components of $sym^3(\pi)$ to be generic and the archimedean components of $sym^3(\pi)$ to be holomorphic discrete series in Theorem~\ref{RS} and Corollary~\ref{RS-classical}. We can switch between generic and non-generic in each component since $sym^3(\pi)$ is of type {\bf(G)}. We consider $\sym^3(\pi_p)$ as the $L$-packet on $\GSp(4,\Q_p)$ whose $L$-parameter is the symmetric cube of the $L$-parameter of $\pi_p$. But, when we study global results in Section~\ref{Section global results}, we make the specific choice of local representations in the L-packet as in Theorem~\ref{RS}.
\end{remark}
We want to find {\it a formula for the level} of the paramodular form $f$ in terms of the data of the given newform $\phi$. Since $a(\sym^3(\pi))=\prod_{p} p^{a(\sym^3(\pi_p))}$, we need to calculate $a(\sym^3(\pi_p))$ for each local representation $\pi_p$ of $\GL(2,\Q_p)$ attached to $\phi$.
\subsection{Local \texorpdfstring{${\rm sym^3}$}{} lifting}
\label{local sym3 section}
Let $K$ be a non-archimedean field of characteristic zero with residual characteristic $p$. Let $\OF_K$ be the ring of integers of $K$, and let $\p$ be the maximal ideal of $\OF_K$. Let $k=\OF_K/\p$ be the residue field of $K$ of order $q$. Let $\Phi \in {\rm Gal}(\bar{K}/K)$ be an inverse Frobenius element, i.e., $\Phi$ induces the inverse of the map $x \rightarrow x^q$ on the residue class field extension $\bar{k}/k$. Let $K^{\text{un}}$ be the maximal unramified extension of $K$ inside $\bar{K}$. Let $W(\bar{K}/K)$ be the Weil group of $K$. Then by definition
\begin{equation}
	W(\bar{K}/K)=\bigsqcup_{n\in \Z} \Phi^n I_K,
\end{equation}
where $I_K$ is the inertia group, which can be identified with the Galois group ${\rm Gal}(\bar{K}/K^{\text{un}})$. A representation of $W(\bar{K}/K)$ is a continuous homomorphism $\varphi:W(\bar{K}/K) \rightarrow \GL(V)$, where $V$ is a finite-dimensional complex vector space. Let $\omega$ be the one-dimensional representation of $W(\bar{K}/K)$ with the property $\omega(I_K)=1$ and $\omega(\Phi)=q^{-1}$. A pair $(\varphi,N)$ is called a ``Weil-Deligne representation" if $\varphi$ is a representation of $W(\bar{K}/K)$ and $N$ is a nilpotent endomorphism of $V$ such that $\varphi(g) N \varphi(g)^{-1}=\omega(g)N$. For general reference in this section, see \cite{Rohrlich1994}. By the local Langlands correspondence, there is a one-to-one correspondence between the irreducible admissible representations $\pi$ of $\GL(2,K)$ and the two-dimensional Weil-Deligne representations $(\varphi,N)$. We refer to $(\varphi,N)$ as the local parameter of $\pi$. Now, there are three types of infinite dimensional irreducible admissible representations of $\GL(2,K)$: 
\begin{enumerate}
\item The principal series representations $\chi_1 \times \chi_2$, where $\chi_1,\chi_2$ are characters of $K^{\times}$. The local parameter is given by
\begin{equation}
\label{localpara for PS}
\varphi(w)= \begin{bmatrix}
 \chi_1(w)&\\
 &\chi_2(w)
 \end{bmatrix}, w \in W(\bar{K}/K)
\text{ and } N=0.
 \end{equation}
\item The (twisted) Steinberg representations $\chi \text{St}_{\GL(2,K)}$, where $\chi$ is a character of $K^{\times}$. The local parameter is given by
\begin{equation} 
\label{localpara for St}
\varphi(w)= \begin{bmatrix}
|w|^{\frac 12}\chi(w)&\\
&|w|^{-\frac12}\chi(w)
\end{bmatrix}, w \in W(\bar{K}/K)
\text{ and }
N=\begin{bmatrix}
0&1\\
0&0
\end{bmatrix}.
\end{equation}
\item The supercuspidal representations. The local parameter is an irreducible representation of $W(\bar{K}/K)$ with $N=0$.
\end{enumerate}
Here we identify the characters of $W(\bar{K}/K)$ and $K^{\times}$ via the Artin isomorphism. For $p \ge 3$, \emph{every} supercuspidal representation of $\GL(2,K)$ has a local parameter of the form
\begin{equation}
\label{local parameter of D.S.C}
\varphi=\ind_{W(\bar{K}/F)}^{W(\bar{K}/K)}(\xi) \text{ with } \xi\neq \xi^{\sigma},
\end{equation}
where $F$ is a quadratic extension of $K$ and $\xi$ is  a character of $W(\bar{K}/F)$. Here, $\sigma$ is the nontrivial element in $W(\bar{K}/K)\setminus W(\bar{K}/F)$ and the conjugate character $\xi^{\sigma}$ of $W(\bar{K}/F)$ is defined by $\xi^{\sigma}(x)=\xi(\sigma x \sigma^{-1}),\ x \in W(\bar{K}/F)$. With respect to a suitable basis, $\varphi$ has the following matrix form,
\begin{equation}
\label{local parameter for sc}
\varphi(x)=\begin{bmatrix}\xi(x) & \\ &\xi^{\sigma}(x)\end{bmatrix},\quad x \in W(\bar{K}/F), \qquad \text{ and }\qquad \varphi(\sigma)=\begin{bmatrix}& 1\\ \xi(\sigma^2)&\end{bmatrix}.
\end{equation}
The representation of $\GL(2,K)$ with a local parameter as in (\ref{local parameter of D.S.C}) is called a \emph{dihedral} supercuspidal representation of $\GL(2,K)$ and we denote it by $\omega_{F, \xi}$, where $\xi$ is a character of $F^{\times}$ corresponding to the character $\xi$ of $W(\bar{K}/F)$ via the Artin isomorphism. The conjugate character $\xi^{\sigma}$ of $W(\bar{K}/F)$ corresponds to a character $\xi^{\sigma}$ of $F^{\times}$ defined by $\xi^{\sigma}(x)=\xi(\sigma(x))$ for $x \in F^{\times}$. 
We have the following formula for the conductor $a(\varphi)$ of $\varphi$ (see VI.2 of \cite{Serre1979}),
\begin{equation}
\label{general_conductor_of_sc_parameter}
a(\varphi)=\text{dim}(\xi)d(F/K)+f(F/K)a(\xi),
\end{equation}
where $\text{dim}(\xi)=1$, $\xi$ being a character; $d(F/K)$ is the discriminant of the field extension $F/K$; $f(F/K)$ is the residue class degree, and $a(\xi)$ is the conductor of $\xi$, i.e., $a(\xi)$ is the smallest $n \in \mathbb{N}$ such that $\xi|_{1+\p^n\OF_F}=1$. Hence, for residue characteristic of $K$ odd, we have
\begin{equation}
\label{conductor_of_sc}
a(\varphi)=\begin{cases}
2a(\xi)\qquad &\text{ if }  F/K \text{ is unramified}\\
1+a(\xi)\qquad &\text{ if }  F/K \text{ is ramified}.
\end{cases}
\end{equation}
The following is a very useful fact about the dihedral supercuspidal representations of $\GL(2,K)$ with trivial central character. We will use this remark quite often in this article.
\begin{remark}
\label{remark that xi sigma=xi inverse}
The central character of the dihedral supercuspidal representation $\omega_{F, \xi}$ of $\GL(2,K)$ is 
$\xi|_{K^{\times}}\cdot \chi_{F/K}$,
where $\chi_{F/K}$ is the quadratic character of $K^{\times}$ associated to the quadratic extension $F/K$ such that $\chi_{F/K}\left({N_{F/K}(F^{\times})}\right)=1$. Here $N_{F/K}$ is the norm map from $F^{\times}$ to $K^{\times}$. If $\omega_{F, \xi}$ has trivial central character, i.e., $\xi|_{K^{\times}}\cdot \chi_{F/K}=1$, then by
evaluating $\xi$ at $N_{F/K}(y)$ for any $y \in F^{\times}$, we get $\xi^{\sigma}=\xi^{-1}$ on $F^{\times}$.
\end{remark}
For a character $\chi$ of $K^{\times}$, the smallest non-negative integer $n$ such that $\chi|_{1+\p^n\OF_K}=1$ is called the \emph{conductor} of $\chi$ and denoted by $a(\chi)$. Also, we denote the $L$-packet with the local parameter $\ind_{W(\overline{\Q}_p/F)}^{W(\overline{\Q}_p/\Q_p)}(\xi)\oplus\ind_{W(\overline{\Q}_p/F)}^{W(\overline{\Q}_p/\Q_p)}(\xi^3)$ as $\omega_{F,\xi^3} \oplus \omega_{F,\xi}$, where $F$ is a quadratic extension of $\Q_p$. When $p \ge 3$, using the local parameters in \eqref{localpara for PS}, \eqref{localpara for St}, \eqref{local parameter of D.S.C}, and the $\sym^3$ map in (\ref{sym^3map}), we describe of the conductor $a(\sym^3(\pi_p))$ of $\sym^3(\pi_p)$ in Table~\ref{gerenal conductor of sym3} for the local representations $\pi_p$ of $\GL(2,\Q_p)$ with trivial central character.
\begin{table}
\caption {Conductor of $\sym^3(\pi_p)$ for local representations $\pi_p$ of $\GL(2,\Q_p)$.}
\label{gerenal conductor of sym3}
\renewcommand{\arraystretch}{1.3}
\renewcommand{\arraycolsep}{.16cm}
\[ \begin{array}{ccccc} 
\toprule
\pi_p&\sym^3(\pi_p)&\text{Cond.\ on } \pi_p&a(\pi_p)&a(\sym^3(\pi_p))\\
\toprule
\chi \times \chi^{-1}& \chi^4\times\chi^2\rtimes\chi^{-3}&&a(\chi) + a(\chi^{-1})&2a(\chi^{-3})+a(\chi)+a(\chi^{-1})\\  
\midrule
\chi \text{St}_{\GL(2,\Q_p)}&   \chi^3 \text{St}_{\GSp(4,\Q_p)}&\chi \text{ is ram.}&2a(\chi) &4a(\chi^3)\\  
\cmidrule{3-5}
\chi^2=1&&\chi \text{ is unr.} &1&3\\  
\midrule
\omega_{F,\xi}& \omega_{F,\xi^3} \oplus \omega_{F,\xi}&F/{\Q_p} \text{ is unr.}&2a(\xi)&2a(\xi^3)+2a(\xi)\\    
\cmidrule{3-5}
&&F/{\Q_p} \text{ is ram.}&a(\xi)+1&a(\xi^3)+a(\xi)+2\\
\bottomrule
 \end{array}\]
\vspace*{-0.2in} 
 \end{table}

Notice that we need a description of the character involved in the $L$-parameter of $\pi_p$ in order to calculate $a(\sym^3(\pi_p))$ which involves the conductor of characters like $a(\chi^3)$ and $a(\xi^3)$.  So, the global question of finding the paramodular level coming from the $\sym^3$ lifting leads to the local question of determining the representation of $\GL(2,\Q_p)$ associated to a cusp form.	In principle the paper \cite{LoefflerWeinstein2012} contains an algorithm to determine the local representations attached to a modular form. However, the local parameter of the output of the algorithm is not always obvious. So, we consider a more specified problem, namely finding the level of the paramodular forms coming from elliptic curves through the $\sym^3$ lifting. Now, we need to find a detailed description of the local representation $\pi_p$ of $\GL(2,\Q_p)$ from the given elliptic curve $E$ over $\Q_p$. We will discuss this in the next section.
\section{Representations attached to elliptic curves over a local field}
\label{Section 2}
In this section we describe the representations associated to elliptic curves over a $p$-adic field. For elliptic curves with potentially multiplicative reduction, one can find the associated representations for any $p$-adic field. If an elliptic curve has additive but potentially good reduction over a local field with residual characteristic $\ge 5$, the associated representation is also known. The main focus of this section is to find a detailed description of the representations attached to elliptic curves with additive but potentially good reduction over a local field of residual characteristic $3$.
\subsection{Elliptic curves over local fields and their reduction types}
Let $K$ be a non-archimedean local field of characteristic zero with residual characteristic $p$. Let $\OF_K$ be the ring of integers of $K$, and let $\p$ be the maximal ideal of $\OF_K$. Let $\varpi_K$ be a uniformizer of $K$. Suppose that $v: K \rightarrow \Z$ is the normalized valuation on $K$, and $k=\OF_K/\p$ is the residue field of $K$ of order $q$. As before, let $\Phi \in {\rm Gal}(\bar{K}/K)$ be an inverse Frobenius element, $K^{\text{un}}$ be the maximal unramified extension of $K$ inside $\bar{K}$, and $I_K$ be the inertia group. Let $E$ be an elliptic curve over $K$ given by a Weierstrass equation of the form
\begin{equation}
\label{W.E of EC}
y^2+a_1xy+a_3y=x^3+a_2x^2+a_4x+a_6
\end{equation}
with the coefficients in $\OF_K$. The discriminant $\Delta$, the $j$-invariant $j(E)$, and the constants $c_4$, $c_6$ are the usual constants attached to (\ref{W.E of EC}) given as follows:
\begin{equation}
\begin{split}
\label{relations between WC}
b_2&=a_1^2+4a_2,\quad b_4=2a_4+a_1a_3,\\
b_6&=a_3^2+4a_6,\quad b_8=a_1^2a_6+4a_2a_6-a_1a_3a_4+a_2a_3^2-a_4^2,\\
c_4&=b_2^2-24b_4,\quad c_6=-b_2^3+36b_2b_4-216b_6,\\
\Delta&=-b_2^2b_8-8b_4^3-27b_6^2+9b_2b_4b_6=\frac{c_4^3-c_6^2}{1728},\\
j(E)&=\frac{c_4^3}{\Delta}.
\end{split}
\end{equation}
Given the minimal Weierstrass equation of $E/K$, we have the following standard results:
\begin{align*}
E \text{ has good (stable) reduction } \quad&\text{ if and only if } \quad v(\Delta)=0.\\
E \text{ has multiplicative (semistable) reduction }\quad&\text{ if and only if } \quad v(\Delta)>0 \text{ and } v(c_4)=0.\\
E \text{ has additive (unstable) reduction }\quad&\text{ if and only if }  \quad v(\Delta)>0 \text{ and } v(c_4)>0.\\
E \text{ has potentially multiplicative reduction }\quad&\text{ if and only if } \quad j(E) \not\in \OF_K.\\
E \text{ has potentially good reduction }\quad&\text{ if and only if } \quad j(E) \in \OF_K.
\end{align*}
If we assume $E/K$ has potentially multiplicative reduction, then $c_4$ and $c_6$ are non-zero since $j(E) \not\in \OF_K$. Then we define the $\gamma$-invariant of $E/K$ by
\begin{equation}
\gamma(E/K)=-\frac{c_4}{c_6} \in K^{\times}/K^{\times2}.
\end{equation}
This quantity is well defined and independent of the choice of the Weierstrass equation. Let $\sigma_E$ be the Weil-Deligne representation associated to $E$ (see sections 13-15 in \cite{Rohrlich1994}), which corresponds to an irreducible admissible representation $\pi_E$ of $\GL(2,K)$ by the local Langlands correspondence. We want a description of the representation $\pi_E$ in terms of the Weierstrass coefficients of $E$. When $E$ has potentially multiplicative reduction, we have the following result from \cite{Rohrlich1994}. 
\begin{theorem}
\label{potentially multiplicative}
Let $E/K$ be an elliptic curve with potentially multiplicative reduction, i.e., $j(E) \not\in \OF_K$. Then the $\GL(2,K)$ representation $\pi_E$ associated to $E$ is given by
$$\pi_E=(\gamma(E/K),\cdot)\St_{\GL(2,K)},$$
where $(\gamma(E/K),\cdot)$ is the quadratic character of $K^{\times}$ defined by the Hilbert symbol $(\cdot,\cdot)$. Furthermore, we have one of the following cases:
\begin{itemize}
	\item $(\gamma(E/K),\cdot)$ is trivial if and only if $E/K$ has split multiplicative reduction.
	\item $(\gamma(E/K),\cdot)$ is non-trivial and unramified if and only if  $E/K$ has non-split multiplicative reduction.
	\item $(\gamma(E/K),\cdot)$ is ramified if and only if  $E/K$ has additive reduction.
\end{itemize}
\end{theorem}
\subsection {The residual characteristic \texorpdfstring{$\ge 5$}{} case}
When $E$ has additive but potentially good reduction and the residual characteristic of $K$ is greater than or equal to $5$, the following results from \cite{Turki2015} describe the representation $\pi_E$.

\begin{theorem}
\label{principal series}
Let $E/K$ be an elliptic curve with additive but potentially good reduction, i.e., $j(E) \in \OF_K$. Assume that $(q-1)v(\Delta) \equiv 0\ ( \text{mod }12)$ and $e=\frac{12}{{\rm gcd}(v(\Delta),12)}$. Then the $\GL(2,K)$ representation $\pi_E$ associated to $E$ is a principal series representation, i.e., $\pi_E=\chi \times \chi^{-1}$, where $\chi$ is a character of $K^{\times}$ satisfying the following properties:
	\begin{enumerate}
		\item[(i)] The conductor $a(\chi)$ of $\chi$ is 1, i.e., $\chi$ is trivial on $1+\p\OF_K$,
		\item[(ii)] $\chi|_{W_{q-1}}$ is trivial on the index $e$-subgroup $W_{q-1}^e$, where $W_{q-1}$ is the group of $(q-1)$th root of unity,
	    \item[(iii)] The character of ${W_{q-1}}/{W_{q-1}^e} \cong \Z/e\Z$ induced by $\chi$ has order $e$.
	    \end{enumerate}
    Furthermore, there is a unique such character when $e=2$, and there are exactly two such characters when $e \in \lbrace 3,4,6\rbrace$ which are inverses of each other.
\end{theorem}

\begin{theorem}
\label{super cuspidal}
	Let $E/K$ be an elliptic curve with additive but potentially good reduction, i.e., $j(E) \in \OF_K$. Assume that $(q-1)v(\Delta) \not\equiv 0 \ ( \text{mod }12)$ and $e=\frac{12}{{\rm gcd}(v(\Delta),12)}$. Then the corresponding $\GL(2,K)$ representation $\pi_E$ is a dihedral supercuspidal representation, i.e., $\pi_E=\omega_{F,\xi}$, where $F$ is the unramified quadratic extension of $K$ and $\xi$ is a character of $F^{\times}$ satisfying the following properties:
	\begin{enumerate}
		\item[(i)] The conductor $a(\xi)$ of $\xi$ is 1, i.e., $\xi$ is trivial on $1+\p\OF_{F}$,
		\item[(ii)] $\xi|_{\OF_{F}^{\times}}$ has order $e$,
		\item[(iii)] $\xi(\varpi_F)=-1$. Here, $\varpi_F$ is a uniformizer of $F$ chosen to be in $K$.
	\end{enumerate}
Furthermore, there are exactly two such characters which are Galois conjugates of each other for each $e$.
\end{theorem}
In the next subsection, we will prove two theorems analogous to Theorem~\ref{principal series} and Theorem~\ref{super cuspidal} for residual characteristic $3$. For residual characteristic greater than $3$, the representation $\pi_E$ coming from $E$ is given in terms of $j(E)$ and $v(\Delta)$. When the residual characteristic of $K$ is equal to $3$, we describe the representation $\pi_E$ attached to $E/K$ in terms of $v(c_4)$, $v(c_6)$, $v(\Delta)$, and some conditions on the underlying field $K$.
\subsection {The residual characteristic \texorpdfstring{$3$}{} case}
Assume that the residual characteristic of $K$ is $3$ and $v(3)=1$. In Table~\ref{IndexNeron}, we define a list of conditions on $E$ in terms of the quantities $v(c_4)$, $v(c_6)$, and $v(\Delta)$. 
This table is reproduced from Table~II of \cite{Papadopoulos1993}. For $i=2$ or $5$, we define a condition ${\rm R}_i$ on $E/K$ by
\begin{equation}
\label{additional condition on EC}
{\rm R}_i:\quad x^3-3c_4x-2c_6\equiv 0 \ {\rm mod}\  (27\varpi_K^i)  \text{ for some } x \in \OF_{K}. 
\end{equation}
If an elliptic curve $E/K$ does not satisfy the condition ${\rm R}_i$, then we denote it by ``non ${\rm R}_i$" in Table~\ref{IndexNeron}. When $K=\Q_3$, $v(c_4)\ge2$, and $v(c_6)=3$, the condition ${\rm R}_2$ is equivalent to
\begin{equation}
(c_6/3^3)^2+2\equiv c_4/3 \ \rm{mod}\  9.
\end{equation}
When $K=\Q_3$, $v(c_4)\ge4$, and $v(c_6)=6$, the condition ${\rm R}_5$ is equivalent to
\begin{equation}
(c_6/3^6)^2+2\equiv c_4/3^3 \ \rm{mod}\  9.
\end{equation}
There are also some conditions on the underlying field $K$ in Table~\ref{IndexNeron}, called ``reducibility condition", that determines whether the Galois representation attached to $E$ is reducible or irreducible. The exponent $v( N)$ of the conductor $N$ of the elliptic curve $E$ appears on the last column of Table~\ref{IndexNeron}. The list of conditions does not depend on $v( N)$. In fact $v( N)$ can be determine in terms of the quantities $v(c_4)$, $v(c_6)$, and $v(\Delta)$. For each condition in Table~\ref{IndexNeron}, we will describe the $\GL(2,K)$ representation $\pi_E$ associated to $E$. The following lemma ensures that this way we get all the possible $\GL(2,K)$ representations associated to an elliptic curve $E/K$ with additive but potentially good reduction.  

Here we assume $v(3)=1$ for simplicity of the conditions in Table~\ref{IndexNeron}. One can also determine the $\GL(2,K)$ representations associated to elliptic curves over $K$ when $v(3)> 1$, the conditions in Table~\ref{IndexNeron} will be more complicated in that case.

\begin{lemma} 
Suppose that $E$ is an elliptic curve over $K$ given by a minimal Weierstrass equation of the form (\ref{W.E of EC}) with the coefficients in $\OF_K$. Let $\Delta$ be the discriminant, and $c_4$, $c_6$ be the usual constants attached to the equation (\ref{W.E of EC}) as defined in (\ref{relations between WC}). Then the following statements are equivalent:
\begin{enumerate}
\item[(i)] $E$ has additive but potentially good reduction.
\item[(ii)] $E$ satisfies one and only one of the conditions in Table~\ref{IndexNeron}.
\end{enumerate}
\end{lemma}
\begin{proof} The Kodaira-N\' eron type of $E/K$ is one of the types described in Theorem~8.2 of \cite{Silverman2009}, and it can be determined in terms of the coefficients of (\ref{W.E of EC}) using the Tate's algorithm. Now, the proof follows from the following equivalent steps:
	
\vspace*{-0.25 in}	
\begin{align*}
E &\text{ has additive but potentially good reduction}.\\
\Leftrightarrow &\ \text{The } j\text{-invariant } j(E) \text{ is integral, i.e., } v\left(j(E)\right)\ge0.\ (\text{See Proposition 5.5 in \cite{Silverman1994}}.)\\
\Leftrightarrow &\  3v\left(c_4\right) \ge v(\Delta).\ (\text{By definition of } j(E) \text{ as in (\ref{relations between WC})}.)\\
\Leftrightarrow &\ \text{The possible Kodaira-N\' eron types of } E \text{ are } {\rm I}_0^{*}, {\rm II}, {\rm II}^{*}, {\rm III}, {\rm III}^{*}, {\rm IV}, {\rm IV}^{*}.\ 
(\text{Table II in \cite{Papadopoulos1993}}.)\\
\Leftrightarrow &\ E \text{ satisfies one and only one of the conditions in Table~\ref{IndexNeron}}.\ (\text{By Table II in \cite{Papadopoulos1993}}.)\qedhere
\end{align*}
\end{proof}

\begin{table}
	\caption{Table of conditions in terms of the quantities $v(\Delta)$, $v(c_4)$, and $v(c_6)$.}
	\label{IndexNeron} 
	\renewcommand{\arraystretch}{0.8}
	\renewcommand{\arraycolsep}{0.23cm}
	\[ \begin{array}{cccccccc} 
	\toprule
	\text{Name of}& \text{Reducibility}&v(\Delta)&v(c_4)&v(c_6)& \text{Additional}&\text{N\'{e}ron}&v(N)\\
	\text{condition}&\text{condition} &&&&\text{condition} &\text{type}&\\
	\toprule
	{\rm P}_2&&6&2&3&&{\rm I}_0^{*}&2\\
	\cmidrule{4-6}
	&&&3&\ge 6&&&\\
	\toprule
	{\rm P}_4& -1 \in k^{\times 2}&3&\ge2&3&{\rm R}_2 \text{ as in }\eqref{additional condition on EC}&{\rm III}& 2\\
	\cmidrule{4-6}
	&&&2&\ge 5&&&\\\cmidrule{3-7}
	&&9&\ge4&6&{\rm R}_5 \text{ as in }\eqref{additional condition on EC}&{\rm III}^{*}&\\
	\cmidrule{4-6}
	&&&4&\ge 8&&&\\
	\toprule
	{\rm S}_4&-1 \not\in k^{\times2 }&3&\ge2&3&{\rm R}_2 \text{ as in }\eqref{additional condition on EC}&{\rm III}& 2\\
	\cmidrule{4-6}
	&&&2&\ge 5&&&\\\cmidrule{3-7}
	&&9&\ge4&6&{\rm R}_5 \text{ as in }\eqref{additional condition on EC}&{\rm III}^{*}&\\
	\cmidrule{4-6}
	&&&4&\ge 8&&&\\
	\toprule
	{\rm P}_3&\Delta \in K^{\times2 }&4&2&3&&{\rm II}&4\\
	\cmidrule{3-7}
	&&12&5&8&&{\rm II}^{*}&\\
	\toprule
	{\rm S}_3& \Delta\not\in K^{\times2 }&4&2&3&&{\rm II}&4\\
	\cmidrule{3-7}
	&&12&5&8&&{\rm II}^{*}&\\
	\toprule
	{\rm P}_6&\Delta \in K^{\times2 }&6&3&5&&{\rm IV}&4\\
	\cmidrule{3-7}
	&&10&4&6&&{\rm IV}^{*}&\\
	\toprule
	{\rm S}_6& \Delta \not\in K^{\times2 }&6&3&5&&{\rm IV}&4\\
	\cmidrule{3-7}
	&&10&4&6&&{\rm IV}^{*}&\\
	\toprule
	{\rm S}^{'}_6&\Delta \not\in K^{\times2 }&3&\ge2&3&{\rm non \ R}_2&{\rm II}&3\\
	\cmidrule{4-6}
	&&&2&4&&&\\
	\cmidrule{3-7}
	&&5&2&3&&{\rm IV}&\\
	\cmidrule{3-7}
	&&9&\ge4&6&{\rm non \ R}_5&{\rm IV}^{*}&\\
	\cmidrule{4-6}
	&&&4&7&&&\\
	\cmidrule{3-7}
	&&11&4&6&&{\rm II}^{*}&\\
	\toprule
	{\rm S}^{''}_6&\Delta \not\in K^{\times2 }&5&\ge3&4&&{\rm II}&5\\
	\cmidrule{3-7}
	&&7&\ge4&5&&{\rm IV}&\\
	\cmidrule{3-7}
	&&11&\ge5&7&&{\rm IV}^{*}&\\
	\cmidrule{3-7}
	&&13&\ge6&8&&{\rm II}^{*}&\\
	\bottomrule
	\end{array}\]
\end{table}

Assume that $E/K$ has additive but potential good reduction, i.e., $E$ satisfies one of the conditions in Table~\ref{IndexNeron}. Let $L/K^{\text{un}}$ be the smallest extension such that $E/L$ has good reduction. By the corollary after Lemme 3 of \cite{Kraus2007}, 
we have 
\begin{equation}
L=K^{\text{un}}(E[2], \Delta^{\frac 14}),
\end{equation}
where $K^{\text{un}}(E[2])$ is the splitting field of the polynomial on the right hand side of (\ref{W.E of EC}). Let $\sigma_E$ be the Weil-Deligne representation associated to $E/K$, which corresponds to an irreducible admissible representation $\pi_E$ of $\GL(2,K)$ by the local Langlands correspondence.
Now the kernel of $\sigma_E:W(\bar{K}/K) \rightarrow \GL(2,\C)$ is ${\rm Gal}(\bar{K}/L)$ (see section 2 of \cite{Rohrlich1993}). So, $\sigma_E:W(L/K) \cong W(\bar{K}/K)/{\rm Gal}(\bar{K}/L) \rightarrow \GL(2,\C)$ is a faithful representation. Let $\Lambda={\rm Gal}(L/K^{\text{un}})$. Then, we have
\begin{equation}
\label{def of W(L/K)}
I_K/{\rm Gal}(\bar{K}/L) \cong \Lambda  \text{ and } W(L/K)=\Lambda \rtimes \left\langle\Phi\right\rangle,
\end{equation}
and 
\begin{equation}
\label{sigmaE on lamda}
\sigma_E(I_K)=\sigma_E(I_K/{\rm Gal}(\bar{K}/L))=\sigma_E(\Lambda) \cong \Lambda.
\end{equation}
\begin{remark}
\label{remark1} 
Since $\sigma_E$ is faithful representation of $W(L/K)$, using (\ref{def of W(L/K)}) and (\ref{sigmaE on lamda}) we conclude that the order of $\sigma_E|_{I_K}$ is $|\Lambda|$. 
\end{remark}
\noindent Also, we will use the following remark in the proofs of the main theorems of this section.
\begin{remark}
	\label{ableian-reducible}
	A faithful two-dimensional semisimple complex representation of a group is reducible if and only if the group is abelian. So, the Weil-Deligne representation $\sigma_E$ associated to the given elliptic curve $E/K$ is reducible if and only if $W(L/K)$ is abelian, i.e., the image of $\sigma_E$ is abelian (also see Proposition 2 in \cite{Rohrlich1993}).	
\end{remark}
\subsection{Principal series representations in residual characteristic \texorpdfstring{$3$}{}}
We continue to assume that the residual characteristic of $K$ is $3$ and $v(3)$=1. The following theorem describes all possible principal series representations of $\GL(2,K)$ coming from elliptic curves over $K$ with additive but potentially good reduction.
\begin{theorem}
\label{thm1}
Let $E/K$ be an elliptic curve given by a minimal Weierstrass equation of the form (\ref{W.E of EC}) with the coefficients in $\OF_K$. Assume that $E$ satisfies one of the conditions in $\lbrace{{\rm P}_2, {\rm P}_4, {\rm P}_3, {\rm P}_6\rbrace}$ as defined in Table \ref{IndexNeron}. Then the corresponding $\GL(2,K)$ representation $\pi_E$ is a principal series representation, i.e., $\pi_E=\chi \times \chi^{-1}$, where $\chi$ is a character of $K^\times$ satisfying the following properties:
\begin{enumerate}
	\item[(i)]
	If $E$ satisfies the condition ${\rm P}_m$,  then $\chi|_{\OF_K^\times}$ has order $m$. Here, $m \in \left\{2,3,4,6\right\}$.
	\item[(ii)] The conductor $	a(\chi)$ of $\chi$ is given by
	\begin{equation*}
	a(\chi)=
	\begin{cases}
	1 \quad \text{if } E \text { satisfies } {\rm P_2} \text{ or }{\rm P_4},\\
	2 \quad \text{if } E \text { satisfies } {\rm P_3} \text{ or }{\rm P_6}.\\
	\end{cases}
	\end{equation*}
	\end{enumerate}
Furthermore, there is a unique such character on $\OF_K^\times$ when $E$ satisfies ${\rm P}_2$, and there are exactly two such characters on $\OF_K^\times$ when $E$ satisfies ${\rm P}_4$, which are inverses of each other.	
\end{theorem}
\noindent Here, we consider the following diagram.
\vspace*{-0.32 in}
\begin{center}
\begin{equation}
\label{diagram for PS general case}
\begin{minipage}{80ex}
\[
\xymatrixrowsep{0.26in}
\xymatrixcolsep{0.9in}
\hspace{1.4 in} \xymatrix{\bar K\ar@{-}[d]\ar@/_1.5pc/@{-}[dd]_{I_K}\\L \ar@{-}[d]\ar@/^1.4pc/@{-}[d]^{\Lambda \phantom{xxxxxxxxxxxxxxxxxxx}}\\K^{\rm un}\ar@{-}[d]\ar@/_1.5pc/@{-}[d]_{\left\langle \Phi \right\rangle}\\K \ar@{-}[d]\\ \Q_3
}
\]
\end{minipage}
\end{equation}
\end{center}
The following lemma is very useful for Theorem~\ref{thm1}, which can easily be seen from the explicit form (\ref{localpara for PS}) of $\sigma_E$ and Remark~\ref{remark1}. 
\begin{lemma} 
	\label{order of chi}
	Assume that $\sigma_E \cong\chi \oplus \chi^{-1}$, where $\chi$ is a character of $W(L/K)$. Then $\chi|_{\Lambda}$ has order $|\Lambda|$, i.e.,\  $n=|\Lambda|$ is the smallest integer $n$ such that $\chi^{n}=1$ on $\Lambda$. 
\end{lemma}

\begin{proof}[\textbf{Proof of Theorem~\ref{thm1}}]
Suppose that $E$ satisfies the condition ${\rm P}_m$ for $m \in \lbrace{2,3,4,6\rbrace}$. By Th\'eor\`eme 1 of \cite{Kraus2007} (also see Theorem 3.1 of \cite{Kobayashi2002}), 
we have $\Lambda \cong \Z/m\Z$. Now, Proposition~3.2 of \cite{Kobayashi2002} implies that $W(L/K)$ is abelian, i.e.,\ the image of $\sigma_E$ is abelian. Then, using Remark \ref{ableian-reducible} (also see Proposition~3.3 of \cite{Kobayashi2002}), we get $\sigma_E\cong\chi \oplus \chi^{-1}$, where $\chi$ is a character of $W(L/K)$. By the local Langlands correspondence, the corresponding $\GL(2,K)$ representation $\pi_E$ is the principal series representation $\chi\times\chi^{-1}$, where $\chi$ is the corresponding character of $K^{\times}$. Using Lemma~\ref{order of chi} and the Artin isomorphism we see that $\chi|_{\OF_K^{\times}}$ has order $m$.
From Table~\ref{IndexNeron}, $a(\pi_E)=v(N)=2$ when $E$ satisfies either ${\rm P_2}$ or ${\rm P_4}$, which implies $a(\chi)=1$. In this case, we get an induced character $\chi: \OF_K^{\times}/(1+\p\OF_K)\cong k^{\times} \rightarrow \C^\times$. Note that $k^\times$ is a cyclic group of order $3^n-1$, where $K$ has degree $n$ over $\Q_3$. 

When $E$ satisfies the condition ${\rm P_2}$, the induced character $\chi: k^{\times} \rightarrow \C^\times$ has order $2$ and $2 \mid 3^n-1$. There is only one element of order $2$ in $k^{\times}$. So, the induced character is the unique such character of order $2$.

When $E$ satisfies the condition ${\rm P_4}$, the induced character $\chi: k^{\times} \rightarrow \C^\times$ has order $4$. Since $-1 \in k^{\times2}$, the quadratic extension $\Q_3(i)$ of $\Q_3$ is contained in $K$. Then $2\mid n$ and $4 \mid 3^{n}-1$. Now, there are exactly $\varphi(4)=2$ elements of order $4$ in $k^{\times}$. So, there are exactly two such characters $\chi$ on $ \OF_K^{\times}$, which are inverses of each other.

Similarly, when $E$ satisfies ${\rm P_3}$ or ${\rm P_6}$, it is easy to check that $a(\chi)=2$. Hence, we proved all the cases of Theorem~\ref{thm1}.
\end{proof}

\noindent The following result is a special but more precise version of Theorem~\ref{thm1} for $K=\Q_3$. Since $-1 \not\in \Q_3^{\times 2}$, an elliptic curve over $\Q_3$ {\it never} satisfies ${\rm P_4}$. 
\begin{corollary}
\label{corollary for Q3 in PS}
Let $E$ be an elliptic curve over $\Q_3$ given by a minimal Weierstrass equation of the form (\ref{W.E of EC}) with the coefficients in $\Z_3$. Assume that $E$ satisfies one of the conditions in $\lbrace{{\rm P}_2, {\rm P}_3, {\rm P}_6\rbrace}$ as defined in Table \ref{IndexNeron}. Then the corresponding $\GL(2,\Q_3)$ representation is a principal series representation, i.e., $\pi_E=\chi \times \chi^{-1}$, where $\chi$ is a character $\Q_3^\times$ satisfying the following properties:
\begin{enumerate}
	\item[(i)]	If $E$ satisfies the condition ${\rm P}_m$,  then $\chi|_{\Z_3^\times}$ has order $m$ with $m \in \left\{2,3,6\right\}$.
	\item[(ii)] The conductor $	a(\chi)$ of $\chi$ is given by
	\begin{equation*}
	a(\chi)=
	\begin{cases}
	1 \quad \text{if } E \text { satisfies } {\rm P_2},\\
	2 \quad \text{if } E \text { satisfies } {\rm P_3} \text{ or } {\rm P_6}.
	\end{cases}
	\end{equation*}
	\item[(iii)]
Furthermore, there is a unique such character on $\Z_3^{\times}$ when  $E$ satisfies ${\rm P}_2$, and there are exactly two such characters on $\Z_3^{\times}$ when $E$ satisfies ${\rm P}_3$ or ${\rm P}_6$, which are inverses of each other.	
\end{enumerate}
\end{corollary}

\begin{proof}
Since this corollary is a special case of Theorem~\ref{thm1} for $K=\Q_3$, most statements follow from Theorem~\ref{thm1} except for the property (iii) when $E$ satisfies ${\rm P}_3$ or ${\rm P}_6$. 

Note that when $E$ satisfies  ${\rm P}_3$ or ${\rm P}_6$, we have $a(\chi)=2$. So, the character $\chi|_{\Z_3^{\times}}$ induces a character on $\Z_3/(1+3^2\Z_3)\cong (\Z/9\Z)^{\times}$. Since $(\Z/9\Z)^{\times}$ is a cyclic group of order 6, there are exactly $\varphi(3)=2$ elements of order $3$ (resp.\ order $6$) in $(\Z/9\Z)^{\times}$. This concludes the property (iii) of the corollary when $E$ satisfies ${\rm P}_3$ or ${\rm P}_6$. 
\end{proof}

\subsection{Supercuspidal representations in residual characteristic \texorpdfstring{$3$}{}}
We continue to assume that the residual characteristic of $K$ is $3$ and $v(3)$=1. First, we state some facts that we use to prove Theorem \ref{thm2}. Let $G$ be a group, and $H$ an index-2 subgroup. All representations of these groups are assumed to be finite-dimensional and complex. Let $\sigma \in G \setminus H$, so that $G=H\cup \sigma H$. If $\xi$ is a representation of $H$, then the conjugate representation $\xi^{\sigma}$ is defined by $\xi^{\sigma}(h)=\xi(\sigma h \sigma^{-1})$. We denote the restriction functor by ${\rm res}_H^G$ and the induction functor by ${\rm ind}_H^G$. Then the following lemma is well known.

\begin{lemma}
\label{index2subgroup}
Let $G$ be a group, and $H$ be an index-2 subgroup of $G$. Let $\chi$ be the unique non-trivial character of $G/H$ and $\varphi$ be an irreducible representation of $G$. Then exactly one of the following statements holds-
\begin{enumerate}
\item $\varphi \not\cong \varphi \otimes \chi$ and ${\rm res}_H^G \varphi$ is irreducible. In that case, ${\rm ind}_H^G({\rm res}_H^G(\varphi))= \varphi \oplus (\varphi \otimes \chi)$.

\item $\varphi \cong \varphi \otimes \chi$ and ${\rm res}_H^G \varphi= \xi \oplus \xi^{\sigma}$, where $\xi$ is a representation of $H$. In that case,  $\xi\not\cong \xi^{\sigma}$, and $\varphi={\rm ind}_H^G(\xi)={\rm ind}_H^G(\xi^{\sigma})$.
\end{enumerate} 
\end{lemma}
\noindent The above lemma is useful when the Galois representation $\sigma_E$ attached to $E$ is irreducible. Also, we use the following remark for Theorem~\ref{thm2}.
\begin{remark}
	\label{remark about order of xi}
	Let $F$ be a quadratic extension of $K$. Assume that $\sigma_E=\text{ind }_{W(L/F)}^{W(L/K)}(\xi)$, where $\xi$ is a character of $W(L/F)$. Since $\sigma_E$ is a faithful representation of $W(L/K)$, using a similar argument as in Remark~\ref{remark1}, the order of $\sigma_E|_{{\rm Gal}(L/F^{\rm un})}$ is $|{\rm Gal}(L/F^{\rm un})|$. Then one can show that $\xi|_{{\rm Gal}(L/F^{\rm un})}$ also has order $|{\rm Gal}(L/F^{\rm un})|$ by looking at the explicit form of $\sigma_E$ as in \eqref{local parameter of D.S.C}. Note that, when $F/K$ is unramified, ${\rm Gal}(L/F^{\rm un})={\rm Gal}(L/K^{\rm un})=\Lambda$.
\end{remark}
\noindent The following theorem describes all the dihedral supercuspidal representations of $\GL(2,K)$ associated to elliptic curves over $K$ with additive but potentially good reduction.
\begin{theorem}
	\label{thm2}
	Let $E/K$ be an elliptic curve given by a minimal Weierstrass equation of the form (\ref{W.E of EC}) with the coefficients in $\OF_K$. Assume that $E$ satisfies one of the conditions in $\lbrace{{\rm S}_4, {\rm S}_3, {\rm S}_6,{\rm S}^{'}_6,{\rm S}^{''}_6 \rbrace}$ as defined in Table \ref{IndexNeron}. Then the associated $\GL(2,K)$ representation $\pi_E$ is a dihedral supercuspidal representation, i.e., $\pi_E=\omega_{F,\xi}$, where $F$ is a quadratic extension of $K$ and $\xi$ is a character of $F^{\times}$ satisfying the following properties:	
\begin{enumerate}
	\item[(i)]
		\begin{equation*}
	F=
	\begin{cases}
	K(i) &\text{is the unramified extension of } K \text{ if } E \text { satisfies } {\rm S_4},\\
	K(\sqrt{\Delta}) &\text{is the unramified extension of } K \text{ if } E \text { satisfies } {\rm S_3} \text{ or } {\rm S_6},\\
	K(\sqrt{\Delta}) &\text{is a ramified extension of } K \text{ if } E \text { satisfies } {\rm S}^{'}_6 \text{ or } {\rm S}^{''}_6.
	\end{cases}
	\end{equation*}
	\item[(ii)] The conductor $a(\xi)$ of $\xi$ is given by
	\begin{equation*}
a(\xi)=
\begin{cases}
1 \quad \text{if } E \text { satisfies } {\rm S}_4,\\
2 \quad \text{if } E \text { satisfies  a condition in } \left\{ {\rm S}_3,\rm {S}_6,{\rm S}^{'}_6\right\},\\
4 \quad \text{if } E \text { satisfies } {\rm S}^{''}_6.
\end{cases}
\end{equation*}
	\item[(iii)] 
If $E$ satisfies the condition ${\rm S}_m$ for $m \in \left\{3,4,6\right\}$, then $\xi|_{\OF_F^\times}$ has order $m$. If $E$ satisfies the condition ${\rm S}^{'}_6$ or ${\rm S}^{''}_6$, then $\xi|_{\OF_F^\times}$ has order $6$.
\item[(iv)]
		$\xi(\varpi_F)=-1$ when $E$ satisfies one of the conditions in $\lbrace{{\rm S}_4, {\rm S}_3, {\rm S}_6\rbrace}$. Here, $\varpi_F$ is a uniformizer of $F$ chosen to be in $K$.
\item[(v)] 
	In all cases, we have $\xi^{\sigma}=\xi^{-1} \text{ on } F^{\times}$.
\end{enumerate}
Furthermore, when $E$ satisfies ${\rm S}_4$, there are exactly two such characters on $\OF_F^{\times}$ and they are Galois conjugates of each other. 
\end{theorem}

\noindent Here, we need to consider the following two different diagrams depending on $F$ being the unramified or a ramified quadratic extension of $K$.
\begin{equation}
\label{field extension in sc case}
\begin{minipage}{80ex}
\[
\xymatrixrowsep{0.26in}
\xymatrixcolsep{0.9in}
\xymatrix{\bar K=\bar F\ar@{-}[d]\ar@/_1.5pc/@{-}[dd]_{I_K=I_{F}}\\L \ar@{-}[d]\ar@/^1.4pc/@{-}[d]^{\Lambda}\\K^{\rm un}={F}^{\rm un}\ar@{-}[d]\ar@/_1.5pc/@{-}[dd]_{\left\langle \Phi \right\rangle}\ar@/^1.4pc/@{-}[d]^{\left\langle \Phi^2 \right\rangle}\\F \ar@{-}[d]\ar@/^1.4pc/@{-}[d]^{\left\langle \sigma \right\rangle}\\ K
} \hspace{15ex}
\xymatrix{\bar K=\bar F\ar@{-}[d]\ar@/_1.5pc/@{-}[dd]_{I_F} \\L \ar@{-}[d]\ar@{-}[d]\ar@/^4.5pc/@{-}[dd]^{\Lambda}\\F^{\rm un}\ar@{-}[d]\ar@{-}[rd]\\K^{\text{un}} \ar@{-}[d]\ar@/_1.4pc/@{-}[d]_{\left\langle \Phi \right\rangle} & F\ar@/^.1pc/@{-}[ld]\\ K
}
\]
\end{minipage}
\end{equation}
\begin{proof}[\textbf{Proof of Theorem~\ref{thm2}}]
We distinguish two cases.

\noindent \textbf{Case 1:} Suppose that $E$ satisfies the condition ${\rm S}_m$ for some $m \in \left\{4,3,6\right\}$. Then by Th\'eor\`eme~1 of \cite{Kraus2007} (also see Theorem~3.1 of \cite{Kobayashi2002}), we have $\Lambda \cong \Z/m\Z$. Here we give a detailed proof for the condition ${\rm S}_4$ and we omit the details for the other two conditions since they are similar. For this case we consider the diagram on the left of (\ref{field extension in sc case}).

If $E$ satisfies the condition ${\rm S}_4$, then $-1 \not\in k^{\times 2}$. By Hensel's lemma $-1 \in \OF_K^{\times} \setminus \OF_K^{\times 2}$. So, the field $F=K(i)$ is the unique unramified extension of $K$ of degree $2$. Now, $W(L/K)=\Z/4\Z  \rtimes \left\langle\Phi\right\rangle$ and $\Phi^2$ acts trivially on $\Z/4\Z$ (since the non-trivial action of $\Phi$ on $\Z/4\Z$ sends $1 \mapsto 3$, $3 \mapsto 1$, and fixes $0,2$). So, we get $W(L/F)=\Z/4\Z  \times \left\langle\Phi^2\right\rangle$.
Note that $\sigma_E$ is a two dimensional faithful representation of $W(L/K)$. Since $-1 \not\in k^{\times 2}$, by Proposition~3.2 of \cite{Kobayashi2002}, $W(L/K)$ is not abelian. Then, by Remark~\ref{ableian-reducible} (also see Proposition~3.3 of \cite{Kobayashi2002}), $\sigma_E$ is an irreducible representation. Now, since $W(L/F)$ is abelian, by Remark \ref{ableian-reducible} and Lemma~\ref{index2subgroup}, $\text{res}_{W(L/F)}^{W(L/K)}(\sigma_E)= \xi \oplus \xi^{\sigma}$ for some character $\xi$ of $W(L/F)$ with $\xi \neq \xi^{\sigma}$ and $\sigma_E=\ind_{W(L/F)}^{W(L/K)}(\xi)$. Here $\sigma$ is the non-trivial element in $W(L/K)\setminus W(L/F)$. 
Hence, by the local Langlands correspondence, $\pi_E=\omega_{F,\xi}$, where $\xi$ is the corresponding character of $F^{\times}$ via the Artin isomorphism. 

The assertion (ii) is immediate from (\ref{conductor_of_sc}). The property (iii) follows from Remark~\ref{remark about order of xi} and the Artin isomorphism, i.e., $\xi|_{\OF_F^{\times}}$ has order $|\Lambda|=|\Z/4\Z|=4$. For (iv), note that the representation $\pi_E$ has trivial central character. By Remark~\ref{remark that xi sigma=xi inverse}, we have $$1=\xi|_{K^{\times}}(\varpi_K).\chi_{F/K}(\varpi_K)=\xi(\varpi_K)\cdot(-1) \quad (\text{since } \chi_{F/K}(\varpi_K)=-1).$$
Since $F$ is the unramified extension over $K$, we may assume $\varpi_K=\varpi_F$. Hence, $\xi(\varpi_F)=-1$. The statement (v) is immediate from Remark~\ref{remark that xi sigma=xi inverse}.

 When $E$ satisfies ${\rm S}_4$, we have a uniqueness result. Since $a(\xi)=1$, we get the induced character $\xi: \OF_{F}^{\times}/(1+\p\OF_F) \rightarrow \C^\times$ of order $4$. Note that $\OF_{F}^{\times}/(1+\p\OF_F)$ is a cyclic group of order $3^{2n}-1$, where $K$ has degree $n$ over $\Q_3$, and $4 \mid 3^{2n}-1$. Now, there are exactly $2$ elements of order $4$ in $\OF_{F}^{\times}/(1+\p\OF_F)$ which are inverses of each other. Since $\xi^{\sigma}=\xi^{-1}$ and $\xi^{\sigma} \neq \xi$, there are exactly two such characters in this case, which are Galois conjugates of each other.

\noindent \textbf{Case 2:} Suppose that $E$ satisfies either the condition ${\rm S}^{'}_6$ or ${\rm S}^{''}_6$. Then by Th\'eor\`eme~1 of \cite{Kraus2007} (also see Theorem 3.1 of \cite{Kobayashi2002}), we have $\Lambda \cong \Z/3\Z \rtimes \Z/4\Z$. We will give a proof for the condition ${\rm S}^{'}_6$; the other case is similar. For this case we consider the diagram on the right of (\ref{field extension in sc case}).
If $E$ satisfies the condition ${\rm S}^{'}_6$, then $\Delta \not\in K^{\times 2}$ and $v(\Delta)$ is odd. So, $\Delta=\varpi^{2l+1}\cdot u$ for some $u \in \OF_K^{\times}$, and 
\begin{equation*}
K(\sqrt{\Delta})=\begin{cases}
K(\sqrt{\varpi})\quad  &\text{if } u \in  \OF_K^{\times 2},\\
K(\sqrt{\varpi u})\quad  &\text{if } u \not \in  \OF_K^{\times 2}.
\end{cases} 
\end{equation*}
In either case, $K(\sqrt{\Delta})$ is a ramified extension of $K$ of degree $2$. By Proposition 3.2 of \cite{Kobayashi2002}, $W(L/K(\sqrt{\Delta}))$ is abelian since $\Delta \in K(\sqrt{\Delta})^{\times}$. Let $F=K(\sqrt{\Delta})$. Since $F^{\rm un}$ is the compositum of $F$ and $K^{\rm un}$, we get ${\rm Gal}(K^{\text{un}}/K) \cong {\rm Gal}(F^{\text{un}}/F)=\left\langle \Phi \right\rangle$. Here we consider an inverse Frobenius $\Phi \in {\rm Gal}(F^{\text{un}}/F)$ as an image of an inverse Frobenius of ${\rm Gal}(\bar{K}/K)$ inside ${\rm Gal}(F^{\text{un}}/F)$. Now, ${\rm Gal}(L/F^{\text{un}}) \cong \Z/6\Z$ is the unique subgroup of order $6$ in $\Lambda$. Hence, we get $W(L/F)= \Z/6\Z \times\left\langle \Phi \right\rangle$.
Then, using a similar argument as in case 1 we get $\pi_E=\omega_{F,\xi}$, where the field $F$ is a ramified quadratic extension of $K$ and $\xi$ is the corresponding character of $F^{\times}$. Also, one can easily prove the properties (ii), (iii), and (v) using similar arguments as in case 1.
\end{proof}
\noindent The following result is a special but more precise version of Theorem~\ref{thm2} for $K=\Q_3$:
\begin{corollary}
\label{corollary for Q3 for sc}
Let $E$ be an elliptic curve over $\Q_3$ given by a minimal Weierstrass equation of the form (\ref{W.E of EC}) with the coefficients in $\Z_3$. Assume that $E$ satisfies one of the conditions in $\lbrace{{\rm S}_4, {\rm S}_3, {\rm S}_6,{\rm S}^{'}_6,{\rm S}^{''}_6 \rbrace}$ as defined in Table~\ref{IndexNeron}. Then the corresponding $\GL(2,\Q_3)$ representation $\pi_E$ is a supercuspidal representation, i.e., $\pi_E=\omega_{F,\xi}$, where  $F$ is a quadratic extension of $\Q_3$ and $\xi$ is a character of $F^{\times}$ with the following properties:
\begin{enumerate}
         \item[(i)]
         \label{quadratic extension F for Q3}
	\begin{equation*}
	F=
	\begin{cases}
	\Q_3(i) &\text{is the unramified extension of } \Q_3 \text{ if } E \text { satisfies } {\rm S_4},\\
	\Q_3(\sqrt{\Delta}) &\text{is the unramified extension of } \Q_3 \text{ if } E \text { satisfies }{\rm S_3} \text{ or } {\rm S_6},\\
	\Q_3(\sqrt{\Delta}) &\text{is a ramified extension of } \Q_3 \text{ if } E \text { satisfies } {\rm S}^{'}_6 \text{ or } {\rm S}^{''}_6.
	\end{cases}
	\end{equation*}
	\item[(ii)] The conductor $a(\xi)$ of $\xi$ is given by
	\begin{equation*}
	a(\xi)=
	\begin{cases}
	1 \quad \text{if } E \text { satisfies } {\rm S}_4,\\
	2 \quad \text{if } E \text { satisfies  a condition in } \left\{ {\rm S}_3,\rm {S}_6,{\rm S}^{'}_6\right\},\\
	4 \quad \text{if } E \text { satisfies } {\rm S}^{''}_6.
	\end{cases}
	\end{equation*}
	\item[(iii)] 
If $E$ satisfies the condition ${\rm S}_m$ for $m \in \left\{3,4,6\right\}$, then $\xi|_{\OF_F^\times}$ has order $m$. If $E$ satisfies the condition ${\rm S}^{'}_6$ or ${\rm S}^{''}_6$, then $\xi|_{\OF_F^\times}$ has order $6$.
	\item[(iv)] For all cases we have $\xi^{\sigma}=\xi^{-1} \text{ on } F^{\times}$.
	\item[(v)]
Furthermore, when $E$ satisfies one of the conditions in $\lbrace{{\rm S}_4, {\rm S}_3, {\rm S}_6,{\rm S}^{'}_6 \rbrace}$, there are exactly two such characters on $\OF_F^{\times}$ which are Galois conjugates of each other. When $E$ satisfies ${\rm S}^{''}_6$, there are exactly six such characters on $\OF_F^{\times}$.
\end{enumerate}
\end{corollary}

\begin{proof}
All the properties of this corollary follow from Theorem \ref{thm2} except for the property (v) when $E$ satisfies one of the conditions in $\lbrace{ {\rm S}_3, {\rm S}_6,{\rm S}^{'}_6,{\rm S}^{''}_6\rbrace}$. We will use the following two lemmas to show the property (v). We omit the elementary proofs of the lemmas.
\begin{lemma}
	\label{general split exact sequence}
	Let $G$ be a group and $H$ be a subgroup of $G$ with order $k$. Let $C$ be a cyclic group of order $n$ such that we have the following exact sequence:
	\begin{equation*} 1\rightarrow H  \rightarrow G\rightarrow C \rightarrow 1.\end{equation*}
	If $\text{gcd }(k,n)=1$, then the above exact sequence splits.
\end{lemma}
\begin{lemma}
\label{split exact squence}
Let $F$ be a non-archimedean local field of characteristic $0$. Then
the following exact sequence 
\begin{equation*}
1\rightarrow (1+\p\OF_F)/(1+\p^2\OF_F)\cong \OF_F/\p\OF_F \xrightarrow{\alpha} \OF_{F}^{\times}/(1+\p^2\OF_F) \xrightarrow{\beta}
\OF_{F}^{\times}/(1+\p\OF_F) \rightarrow 1\end{equation*}
splits, i.e., $\OF_{F}^{\times}/(1+\p^2\OF_{F}) \cong (1+\p\OF_F)/(1+\p^2\OF_{F}) \times \OF_{F}^{\times}/(1+\p\OF_F)$.
\end{lemma}
\noindent Now, we will prove the property (v) of Corollary~\ref{corollary for Q3 for sc} in three cases.

\noindent \textbf{Case 1:} Assume that $E$ satisfies ${\rm S}_3$ or ${\rm S}_6$. Then $F=\Q_3(\sqrt{\Delta})$ is the unramified quadratic extension of $\Q_3$. Also, $a(\xi)=2$ and the order of $\xi|_{\OF_F^{\times}}$ is $3$ or $6$.

Now, $(1+\p\OF_F)/(1+\p^2\OF_{F}) \cong \OF_{F}/\p\OF_F$ has order $3^{2}$. Since characteristic of $F$ is $3$, we get
$3x=0$ for all $x \in \OF_{F}/\p\OF_F$. Then $(1+\p\OF_F)/(1+\p^2\OF_{F}) \cong  \OF_{F}/\p\OF_F \cong (\Z/3\Z)^{2}$. Also, $\OF_{F}^{\times}/(1+\p\OF_F)$ is the cyclic group of order $8$. Using Lemma~\ref{split exact squence}, we get
\begin{equation}
\OF_{F}^{\times}/(1+\p^2\OF_{F}) \cong (\Z/3\Z)^{2} \times \Z/8\Z.
\end{equation}
Since $a(\xi)=2$, we get an induced character $\xi:\OF_{F}^{\times}/(1+\p^2\OF_{F})\longrightarrow \C^{\times}$.
Again, using Lemma~\ref{split exact squence}, $\OF_K^{\times}/(1+\p^2\OF_{K})\cong \Z/3\Z \times \Z/2\Z$. Now, $\OF_K^{\times}/(1+\p^2\OF_{K})\hookrightarrow \OF_F^{\times}/(1+\p^2\OF_{F})$ (because $\OF_K^{\times}\cap(1+\p^2\OF_{F})=(1+\p^2\OF_{K})$). Since $\chi_{F/K}$ is the unramified quadratic character of $K^{\times}$, using Remark~\ref{remark that xi sigma=xi inverse}, $ \xi|_{\OF_K^{\times}/(1+\p^2\OF_{K})}=1$.
Then, we get the induced character
\begin{equation}
\xi: \left(\OF_{F}^{\times}/(1+\p^2\OF_{F})\right) / \left(\OF_K^{\times}/(1+\p^2\OF_K)\right)\cong \Z/12\Z \longrightarrow \C^{\times}
\end{equation}
of order $3$ or $6$ when $E$ satisfies ${\rm S_3}$ or ${\rm S_6}$ respectively. 
Now, $\Z/12\Z$ has exactly $2$ elements of order $3$ (resp.\  order $6$), which are inverses of each other. So, using Remark~\ref{remark that xi sigma=xi inverse}, we conclude that there are exactly two such characters of order $3$ (resp.\ order $6$) on $\OF_{F}^{\times}$ when $E$ satisfies ${\rm S_3}$ (resp.\  ${\rm S_6}$), and they are Galois conjugates of each other.

\noindent \textbf{Case 2:} Assume that $E$ satisfies ${\rm S}^{'}_6$. Then $F=\Q_3(\sqrt{\Delta})$ is a ramified quadratic extension of $\Q_3$, the order of $\xi|_{\OF_F^{\times}}$ is $6$, and $a(\xi)=2$. Now, we have
\begin{equation}
\OF_F/\p\OF_F \cong \Z_3/3\Z_3 \cong \Z/3\Z \text{ and } \OF_{F}^{\times}/(1+\p\OF_F) \cong (\Z_3/3\Z_3)^{\times} \cong \Z/2\Z.
\end{equation}
So, by Lemma~\ref{split exact squence}, $\OF_{F}^{\times}/(1+\p^2\OF_{F}) \cong \Z/3\Z \times \Z/2\Z \cong \Z/6\Z$. Since $a(\xi)=2$, we get the induced character
\begin{equation}
\xi:\OF_{F}^{\times}/(1+\p^2\OF_{F}) \cong  \Z/6\Z \longrightarrow \C^{\times}
\end{equation}
of order $6$. There are exactly two elements of order $6$ in $\Z/6\Z$, which are inverses of each other. Hence, using Remark~\ref{remark that xi sigma=xi inverse}, there are exactly two such characters of order $6$ on $\OF_{F}^{\times}$ when $E$ satisfies ${\rm S}^{'}_6$, and they are Galois conjugates of each other.

\noindent \textbf{Case 3:} Assume that $E$ satisfies ${\rm S}^{''}_6$. Then $F=\Q_3(\sqrt{\Delta})$ is a ramified quadratic extension of $\Q_3$, the order of $\xi|_{\OF_F^{\times}}$ is $6$, and $a(\xi)=4$. So, we get the induced character
\begin{equation*}
\xi:\OF_{F}^{\times}/(1+\p^4\OF_{F})\longrightarrow \C^{\times}.
\end{equation*}
Using Lemma \ref{general split exact sequence}, one can show that the following exact sequence 
\begin{equation*}
1\rightarrow (1+\p\OF_F)/(1+\p^4\OF_F)\rightarrow \OF_{F}^{\times}/(1+\p^4\OF_F) \rightarrow\OF_{F}^{\times}/(1+\p\OF_F) \rightarrow 1
\end{equation*} 
splits, i.e.,
\begin{equation}
\label{spliting for p^4}
\OF_{F}^{\times}/(1+\p^4\OF_F) \cong \OF_{F}^{\times}/(1+\p\OF_F) \times  (1+\p\OF_F)/(1+\p^4\OF_F).
\end{equation}
The following lemma gives the structure of the group $(1+\p\OF_F)/(1+\p^4\OF_F)$. 
\begin{lemma}
\label{structure of 1+pF4}
Suppose that $E/\Q_3$ satisfies ${\rm S}^{''}_6$. Then the representation $\pi_E$ of $\GL(2,\Q_3)$ associated to $E/\Q_3$ is $\pi_E=\omega_{F,\xi}$ with $F=\Q_3(\sqrt{-3})$. Moreover, \begin{equation}
(1+\p\OF_F)/(1+\p^4\OF_F) \cong	(\Z/3\Z)^3.
\end{equation}
\end{lemma}
\begin{proof}
When $E$ satisfies ${\rm S}^{''}_6$, we have $\pi_E=\omega_{F,\xi}$ with $F=\Q_3(\sqrt{\Delta})$. To see that $F=\Q_3(\sqrt{-3})$, we need to consider all possible values of $\left( v_3(\Delta),v_3(c_4),v_3(c_6)\right)$ in Table~\ref{IndexNeron} when $E$ satisfies ${\rm S}^{''}_6$, and use the relation $\Delta=\frac{c_4^3-c_6^2}{1728}$. One can easily check that $$3^{-v_3(\Delta)}\Delta\equiv -1 \pmod 3$$ when $E$ satisfies ${\rm S}^{''}_6$. Hence, $F=\Q_3(\sqrt{\Delta}) \cong \Q_3(\sqrt{-3})$. Now, by counting the number of elements of order $3$ in $(1+\p\OF_F)/(1+\p^4\OF_F)$, one can show that $(1+\p\OF_F)/(1+\p^4\OF_F) \cong	(\Z/3\Z)^3$.
\end{proof}
\noindent Since $\OF_{F}^{\times}/(1+\p\OF_F)\cong \Z/2\Z$, using \eqref{spliting for p^4} and Lemma~\ref{structure of 1+pF4}, we get $\OF_{F}^{\times}/(1+\p^4\OF_F) \cong \Z/2\Z \times(\Z/3\Z)^3$. Now, $(1+\p\OF_K)/(1+\p^2\OF_K) \hookrightarrow (1+\p\OF_F)/(1+\p^4\OF_F)$ since $\p\OF_K \cap \p^4\OF_F=\p^2\OF_K$. By corollary~$3$ of \S V3 in \cite{Serre1979}, $1+\p\OF_K=N_{F/K}(1+\p^2\OF_F)$. Since $\xi^{\sigma}=\xi^{-1}$ on $F^{\times}$, we get  $\xi(1+\p\OF_K)=\xi\left(N_{F/K}(1+\p^2\OF_F)\right)=1$. So, $\xi|_{(1+\p\OF_K)/(1+\p^2\OF_K)}=1$, where $(1+\p\OF_K)/(1+\p^2\OF_K)\cong \Z/3\Z$. Then we get the following induced character of order $6$:
\begin{equation*}
\xi: \left(\OF_{F}^{\times}/(1+\p^4\OF_F)\right)/\left((1+\p\OF_K)/(1+\p^2\OF_K)\right)\cong \Z/2\Z \times(\Z/3\Z)^2 \longrightarrow \C^{\times}.
\end{equation*}
Also, $\xi$ is nontrivial on $(1+\p^3\OF_F)/(1+\p^4\OF_F)$ since $a(\xi)=4$. So, there are exactly $6$ characters $\xi$ such that $a(\xi)=4$, the order of $\xi|_{\OF_F^{\times}}$ is $6$, and $\xi^{\sigma}=\xi^{-1}$ on $F^{\times}$. This completes the proof of the property (v) of Corollary \ref{corollary for Q3 for sc}.
\end{proof}
\section{Paramodular forms associated to elliptic curves via the \texorpdfstring{$\sym^3$}{} lifting}
\label{Section 3}
In this section, we will find a formula for the level of the paramodular forms obtained by the $\sym^3$ lifting from non-CM elliptic curves over $\Q$. Let $E$ be a non-CM elliptic curve over $\Q$. Note that $E/\Q$ has the global minimal Weierstrass equation of the form (\ref{W.E of EC}) with the coefficients in $\Z$. The discriminant $\Delta$, the $j$-invariant $j(E)$, and the quantities $c_4$, $c_6$ are the usual constants given by (\ref{relations between WC}). We can consider the same elliptic curve $E$ over $\Q_p$ for each prime $p$. There is a cuspidal automorphic representation $\pi=\otimes_p \pi_p$ of $\GL(2,\A_{\Q})$ associated to $E/\Q$ such that $\pi_p$ is the local representation of $\GL(2,\Q_p)$ attached to $E/\Q_p$. Here we have the following diagram:

\begin{tikzcd}
\textcolor{blue}{E/\Q_p}\arrow[d,blue]&E/\Q \text{ of conductor } N \arrow[r, "sym^3", rightsquigarrow,red]  \arrow[d] \arrow[l,blue]& f \in S_{3}(K(\textcolor{red}{M})) \\
\textcolor{blue}{\pi_p}\arrow[r,blue]&\pi=\bigotimes\limits_{p} \pi_p \text{ on }\GL(2,\A_{\Q})\arrow[r] &  \sym^3(\pi)=\bigotimes\limits_{p}  \sym^3(\pi_p) \text{ on }\GSp(4,\A_{\Q}) \arrow[u]
\end{tikzcd}

\noindent Then using similar arguments as in Corollary~\ref{RS-classical}, $E/\Q$ lifts to a paramodular newform $f$ of level $M=a(\sym^3(\pi))=\prod_{p} p^{a(\sym^3(\pi_p))}$. Our goal is to find $M$ in terms of the Weierstrass coefficients of the given elliptic curve $E/\Q$. 

\subsection{The \texorpdfstring{$\sym^3$}{} lift of the local representations attached to elliptic curves}
\label{sym3 of local reprepesentations}
Before discussing the global results on the $\sym^3$ lifting of $E/\Q$, we will look at the $\sym^3$ lifting of the local representation $\pi_p$ of $\GL(2,\Q_p)$ attached to $E/\Q_p$. We have briefly discussed the local $\sym^3$ lifting in Section~\ref{local sym3 section}. In this section, we specifically study the $L$-packet $\sym^3(\pi_p)$ for a local representation $\pi_p$ attached to an elliptic curve $E/\Q_p$ in detail. Even though we describe the results in this section for $\Q_p$, the same results hold for any non-achimedean local field $K$ of characteristic $0$. First, we will compute $a(\sym^3(\pi_p))$ in Tables~\ref{conductor of sym^3 for potentially multiplicative},  \ref{conductor of sym^3 for p>3} and \ref{conductor of sym^3 for p=3}. Let $v_p$ be the $p$-adic valuation on $\Q_p$.
\begin{theorem}
	Let $p$ be any prime and $E$ be an elliptic curve over $\Q_p$ given by a minimal Weierstrass equation of the form (\ref{W.E of EC}) with coefficients in $\Z_p$. Let $\pi_p$ be the irreducible, admissible representation of $\GL(2,\Q_p)$ attached $E$. 
	\begin{enumerate}
		\item[(i)] If $E$ has potentially multiplicative reduction, then $a(\sym^3(\pi_p))$ is given in Table~\ref{conductor of sym^3 for potentially multiplicative}.
		\item[(ii)] If $E$ has additive but potentially good reduction at $p$ for $p\ge 5$, then $a(\sym^3(\pi_p))$ is given in Table~\ref{conductor of sym^3 for p>3}.
		\item[(iii)] If $E$ has additive but potentially good reduction at $p=3$, then $a(\sym^3(\pi_3))$ is given in Table~\ref{conductor of sym^3 for p=3}.
	\end{enumerate} 
\end{theorem}
\begin{proof}
	
(i) Assume that $E$ has potentially multiplicative reduction, i.e, $j(E) \not\in \Z_p$. By Theorem~\ref{potentially multiplicative}, the representation $\pi_p$ of $\GL(2,\Q_p)$ attached to $E$ is $\pi_p=(\gamma(E/\Q_p),\cdot)\St_{\GL(2,\Q_p)}$. Using the $\sym^3$ map in (\ref{sym^3map}) and the local parameter of $\pi_p$ in (\ref{localpara for St}) we get $$\sym^3(\pi_p)=(\gamma(E/\Q_p),\cdot)\St_{\GSp(4,\Q_p)},$$ and the conductor $a(\sym^3(\pi_p))$ of $\sym^3(\pi_p)$ is given by
\begin{equation*}
a(\sym^3(\pi_p))=\begin{cases}
4a((\gamma(E/\Q_p),\cdot)) \quad &\text{ if } (\gamma(E/\Q_p),\cdot) \text{ is ramified},\\
3 \quad &\text{ if } (\gamma(E/\Q_p),\cdot) \text{ is unramified}.
\end{cases}
\end{equation*}
When $(\gamma(E/\Q_p),\cdot)$ is ramified, it is well known that
\begin{equation*}
a\left((\gamma(E/\Q_p),\cdot)\right)=\begin{cases}
1\quad &\text{ for } p\ge 3,\\
2\text{ or } 3 \quad &\text{ for } p=2.
\end{cases}
\end{equation*}
Hence, we obtain the column of $a(\sym^3(\pi_p))$ in Table~\ref{conductor of sym^3 for potentially multiplicative}.

\vspace{0.1 in}
\noindent (ii) Let $e=\frac{12}{\text{gcd}(12,v_p(\Delta))}$. Assume that $E$ has additive but potentially good reduction, i.e, $j(E) \in \Z_p$.  Then we consider the following two cases.

\noindent \textbf{Case 1:} Suppose that $(q-1)v_p(\Delta) \equiv 0 \pmod{12}$. By Theorem~\ref{principal series}, the corresponding $\GL(2,\Q_p)$ representation $\pi_p$ is of the form $\pi_p=\chi \times \chi^{-1}$. Then, using the $\sym^3$ map in (\ref{sym^3map}) and the local parameter of $\pi_p$ in (\ref{localpara for PS}), we get $\sym^3(\pi_p)=\chi^4 \times \chi^2 \rtimes \chi^{-3}$ and the conductor $a(\sym^3(\pi_p))$ of $\sym^3(\pi_p)$ is given by
	\begin{equation}
	\label{conductor11}
	a(\sym^3(\pi_p))=2a(\chi^3)+2a(\chi).
	\end{equation}
By Theorem~\ref{principal series}, in this case $a(\chi)=1$. Using the fact that $\chi^e$ is unramified, one verifies easily that \eqref{conductor11} simplifies to the values given in Table~\ref{conductor of sym^3 for p>3}.
 
\noindent \textbf{Case 2:} Suppose that $(q-1)v_p(\Delta) \not\equiv 0  \pmod{12}$. By Theorem~\ref{super cuspidal}, the corresponding $\GL(2,\Q_p)$ representation $\pi_p$ is of the form $\pi_p=\omega_{F,\xi}$, where $F$ is the unramified quadratic extension of $\Q_p$ and $\xi$ is a character of $F^{\times}$. Then, using the $\sym^3$ map in (\ref{sym^3map}), the local parameter of $\pi_p$ as in (\ref{local parameter of D.S.C}), and Remark~\ref{remark that xi sigma=xi inverse}, the local parameter of $\sym^3(\pi_p)$ is \begin{equation}
\label{local para of sym3 of D.S.C}
\ind_{W(\overline{\Q}_p/F)}^{W(\overline{\Q}_p/\Q_p)}(\xi^3) \oplus \ind_{W(\overline{\Q}_p/F)}^{W(\overline{\Q}_p/\Q_p)}(\xi).
\end{equation}
Then, using (\ref{conductor_of_sc}) we get 
\begin{equation}
\label{conductor2}
a(\sym^3(\pi_p))=2a(\xi^3)+2a(\xi). 
\end{equation}
Also, in this case, $a(\xi)=1$ by Theorem~\ref{super cuspidal}. Using the fact that $\xi^e$ is unramified, one verifies easily that \eqref{conductor2} simplifies to the values given in Table~\ref{conductor of sym^3 for p>3}.

\vspace{0.1 in}
\noindent (iii) The proof of this part is similar to part (ii), so we will skip the details. Here also we consider two different cases. 

\noindent \textbf{Case 1:} By Corollary~\ref{corollary for Q3 in PS}, when $E$ satisfies one of the conditions in $\lbrace{{\rm P}_2, {\rm P}_3, {\rm P}_6\rbrace}$, the $\GL(2,\Q_3)$ representation $\pi_3$ is of the form $\pi_3=\chi \times \chi^{-1}$. Then, using similar arguments as in case 1 of part (ii), we get 
$a(\sym^3(\pi_3))=4$ when $E$ satisfies ${\rm P}_2$ or ${\rm P}_3$, and $a(\sym^3(\pi_3))=6$ when $E$ satisfies the condition ${\rm P}_6$.

\noindent \textbf{Case 2:} By Corollary~\ref{corollary for Q3 for sc}, when $E$ satisfies one of the conditions in $\lbrace{{\rm S}_4, {\rm S}_3, {\rm S}_6,{\rm S}^{'}_6,{\rm S}^{''}_6 \rbrace}$, the corresponding $\GL(2,\Q_3)$ representation $\pi_3$ is of the form $\pi_3=\omega_{F,\xi}$. Then, using the description of $F$ and $\xi$ from Corollary~\ref{corollary for Q3 for sc} and applying similar reasoning as in case 2 of part (ii), we obtain the column of $a(\sym^3(\pi_3))$ in Table~\ref{conductor of sym^3 for p=3}.
\end{proof}

\begin{table}
\caption {Conductor of $\sym^3(\pi_p)$ when $E/\Q_p$ has potentially multiplicative reduction. Here $\gamma=\gamma(E/\Q_p)$ is the $\gamma$-invariant of $E/\Q_p$ and $(\gamma,\cdot)$ is the Hilbert symbol.}
\label{conductor of sym^3 for potentially multiplicative}
\renewcommand{\arraystretch}{1.5}
\renewcommand{\arraycolsep}{0.24cm}
\[ \begin{array}{ ccccccc } 
\toprule
\text{Condition}&\GL(2,\Q_p)& \sym^3(\pi_p)& \text{Condition}&\text{Prime}&a(\pi_p)&a(\sym^3(\pi_p))\\[-1.1ex]
\text{on } E/\Q_p&\text{rep. }\pi_p&&\text{on }\pi_p &p &&\\
\toprule
 j(E) \not\in \Z_p&(\gamma,\cdot)\text{St}_{\GL(2)}&(\gamma,\cdot)\text{St}_{\GSp(4)}&\multirow{2}{*} { \begin{minipage}{11ex}
 	\vspace*{-0.07in}
 	\begin{center}
 	$(\gamma,\cdot)$
 	$\text{ is ram.}$
 	\end{center}
 	\end{minipage}}&\ge 3&2&4\\  
 \cmidrule{5-7}
&&&&2&4&8\\ 
&&&&&6&12\\ 
\cmidrule{4-7}
&&&\multirow{2}{*} { \begin{minipage}{11ex}
	\vspace*{-0.07in}
	\begin{center}
	$(\gamma,\cdot)$
	$\text{ is urm.}$
	\end{center}
	\end{minipage}}&&1&3\\[3ex]
\bottomrule
\end{array}\]
\caption {Conductor of $\sym^3(\pi_p)$ with $p\ge 5$ when $E/\Q_p$ has additive but potentially good reduction. Here we denote $\ind_{W(\overline{\Q}_p/F)}^{W(\overline{\Q}_p/\Q_p)}(\xi)$ as $\ind_{F}^{\Q_p} \xi$.}
\label{conductor of sym^3 for p>3}
\renewcommand{\arraystretch}{1.5}
\renewcommand{\arraycolsep}{0.15cm}
\[ \begin{array}{ cccccc } 
\toprule
\text{Condition} &\GL(2,\Q_p)&\text{Local parameter} & \text{Condition} &a(\pi_p)&a(\sym^3(\pi_p)) \\[-1.1ex]
\text{on } E/\Q_p&\text{rep. }\pi_p &\text{of }\sym^3(\pi_p) &\text{on } \pi_p& &\\
\toprule
\multirow{2}{*} { \noindent \begin{minipage}{20ex}
	\vspace*{-0.03 in}
	\begin{center}
	$j(E) \in \Z_p $
	$\scriptstyle{(p-1)v_p(\Delta) \equiv 0\ \text{mod } 12}$
	\end{center}
	\end{minipage}}
& \multirow{2}{*} {\begin{minipage}{9ex}
	\vspace*{-0.09 in}
	\begin{center}
	$\chi \times \chi^{-1}$
	$a(\chi)=1$
	\end{center}
	\end{minipage}}
&\chi^3\oplus \chi\oplus\chi^{-1}\oplus\chi^{-3} &\chi|_{\Z_p^{\times}}^2=1&2&4\\ 
\cmidrule{4-6}
&&&\chi|_{\Z_p^{\times}}^3=1&2&2\\ 
\cmidrule{4-6}
&&&\chi|_{\Z_p^{\times}}^4=1&2&4\\  
\cmidrule{4-6}
&&&\chi|_{\Z_p^{\times}}^6=1&2&4\\  
\toprule
 \multirow{3}{*} {\begin{minipage}{20ex}
	\vspace*{-0.27 in}
	\begin{center}
	$j(E) \in \Z_p $
	$\scriptstyle{(p-1)v_p(\Delta) \not\equiv 0\ \text{mod } 12}$
	\end{center}
	\end{minipage}}&
\multirow{3}{*} { \begin{minipage}{11ex}
	\vspace*{0.08in}
	\begin{center}
	$\pi= \omega_{F,\xi}$
	$\scriptstyle{F/\Q_p\text{ is unr.}}$
	$\scriptstyle{a(\xi)=1}$\\
	$\scriptstyle{\xi(\varpi_F)=1}$
	\end{center}
	\end{minipage}}
&\ind_F^{\Q_p} \xi \oplus \ind_F^{\Q_p} \xi^3&\xi|_{\OF_F^{\times}}^3=1&2&2\\  
\cmidrule{4-6}
&&&\xi|_{\OF_F^{\times}}^4=1&2&4\\
\cmidrule{4-6}   
&&&\xi|_{\OF_F^{\times}}^6=1&2&4\\ 
\bottomrule
\end{array}\]
\end{table}

\begin{table}
\caption {Conductor of $\sym^3(\pi_3)$ when $E/\Q_3$ has additive but potentially good reduction. Here we denote $\ind_{W(\overline{\Q}_3/F)}^{W(\overline{\Q}_3/\Q_3)}(\xi)$ as $\ind_{F}^{\Q_3} \xi$.}
\label{conductor of sym^3 for p=3}
\renewcommand{\arraystretch}{1.2}
\renewcommand{\arraycolsep}{0.25cm}
\[ \begin{array}{ cccccc } 
	\toprule
	\text{Condition} & \GL(2,\Q_3) & \text{Local }&\text{Condition}& a(\pi_3)& a(\sym^3(\pi_3)) \\[-0.3ex]
		\text{from}& \text{rep.}& \text{parameter}&\text{on }&&\\[-0.2ex]
        \text{Table }\ref{IndexNeron} &\pi_3& \text{of }\sym^3(\pi_3)&\pi_3&&\\
		  \toprule
{\rm P}_2& \chi \times \chi^{-1} & \chi^3\oplus \chi\oplus\chi^{-1}\oplus\chi^{-3} & a(\chi)=1 & 2  & 4\\
&&&\chi|_{\Z_3^{\times}}^2=1&&\\
\midrule
{\rm P}_3& \chi \times \chi^{-1}& \chi^3\oplus \chi\oplus\chi^{-1}\oplus\chi^{-3} &a(\chi)=2&4&4\\
&&&\chi|_{\Z_3^{\times}}^3=1&&\\
\midrule
{\rm P}_6& \chi \times \chi^{-1}&\chi^3\oplus \chi\oplus\chi^{-1}\oplus\chi^{-3}  &a(\chi)=2&4&6\\
&&&\chi|_{\Z_3^{\times}}^6=1&&\\
\midrule 
{\rm S}_4&\omega_{F,\xi}&\ind_F^{\Q_3} \xi \oplus \ind_F^{\Q_3} \xi^3&a(\xi)=1&2&4\\
&{\scriptstyle F = \Q_3(i)}&&\xi(\varpi_F)=-1&&\\
&\text{(unramified)}&&\xi|_{\OF_F^\times}^4=1&&\\
\midrule
{\rm S}_3& \pi=\omega_{F,\xi}&\ind_F^{\Q_3} \xi \oplus \ind_F^{\Q_3} \xi^3& a(\xi)=2&4&4\\
&{\scriptstyle F={\Q_3}(\sqrt{\Delta})}&&\xi(\varpi_F)=-1&&\\
&\text{(unramified)}&&\xi|_{\OF_F^\times}^3=1&&\\
\midrule
{\rm S}_6&\pi=\omega_{F,\xi}&\ind_F^{\Q_3} \xi \oplus \ind_F^{\Q_3} \xi^3&a(\xi)=2&4&6\\
&{\scriptstyle F={\Q_3}(\sqrt{\Delta})}&&\xi(\varpi_F)=-1&&\\
&\text{(unramified)}&&\xi|_{\OF_F^\times}^6=1&&\\
\midrule
{\rm S}^{'}_6&\pi=\omega_{F,\xi}&\ind_F^{\Q_3} \xi \oplus \ind_F^{\Q_3} \xi^3&a(\xi)=2&3&5\\
&{\scriptstyle F={\Q_3}(\sqrt{\Delta})}&&\xi|_{\OF_F^\times}^6=1&&\\
&\raisebox{0.2 in}{\text{(ramified)}}&&&&\\[-0.2in]
\midrule
{\rm S}^{''}_6&\pi=\omega_{F,\xi}&\ind_F^{\Q_3} \xi \oplus \ind_F^{\Q_3} \xi^3&a(\xi)=4&5&7\\
&{\scriptstyle F={\Q_3}(\sqrt{\Delta})}&&\xi|_{\OF_F^\times}^6=1&&\\
&\raisebox{0.2 in}{\text{(ramified)}}&&&&\\[-0.2in]
\bottomrule
\end{array}\]
\end{table}
\noindent \textbf{Other data of the $L$-packet $\sym^3(\pi_p)$:}

\noindent Now, we will describe the representation type, the $\varepsilon$-factor, and the spin $L$-factor of the $L$-packet $\sym^3(\pi_p)$ in Table~\ref{table of local data for p>5} and Table~\ref{table of local data for p=3}. It is easy to find the local parameter of $\sym^3(\pi_p)$ using the local parameter of $\pi_p$ as in \eqref{localpara for PS}, \eqref{localpara for St}, \eqref{local parameter of D.S.C}, and the $\sym^3$ map in \eqref{sym^3map}. Then, one can determine the representation type of $\sym^3(\pi_p)$ using its local parameter and Table~A.7.\ of \cite{RobertsSchmidt2007}. Here we will discuss the non-trivial case when $\pi_p$ is a dihedral supercuspidal representation, i.e., $\pi_p=\omega_{F,\xi}$. Note that $\ind_{W(\overline{\Q}_p/F)}^{W(\overline{\Q}_p/\Q_p)}(\xi)$ in \eqref{local para of sym3 of D.S.C} is always irreducible. So, we need to check whether $\ind_{W(\overline{\Q}_p/F)}^{W(\overline{\Q}_p/\Q_p)}(\xi^3)$ is reducible or irreducible. We have the following three cases:
\begin{enumerate}
\item[(1)] $\ind_{W(\overline{\Q}_p/F)}^{W(\overline{\Q}_p/\Q_p)}(\xi^3)$ is irreducible and isomorphic to $ \ind_{W(\overline{\Q}_p/F)}^{W(\overline{\Q}_p/\Q_p)}(\xi)$. This happens if and only if $\xi^3=\xi^{\sigma}$ on $F^{\times}$, i.e., $\xi^4=1$.
\item[(2)] $\ind_{W(\overline{\Q}_p/F)}^{W(\overline{\Q}_p/\Q_p)}(\xi^3)$ is reducible. This happens if and only if $\xi^3=(\xi^3)^{\sigma}$ on $F^{\times}$, i.e., $\xi^6=1$.
\item[(3)] $\ind_{W(\overline{\Q}_p/F)}^{W(\overline{\Q}_p/\Q_p)}(\xi^3)$ is irreducible and not isomorphic to $ \ind_{W(\overline{\Q}_p/F)}^{W(\overline{\Q}_p/\Q_p)}(\xi)$. This happens if and only if $\xi^3\neq \xi^{\sigma}$ and $\xi^3\neq(\xi^3)^{\sigma}$ on $F^{\times}$, i.e., $\xi^4\neq 1$ and $\xi^6\neq 1$.
\end{enumerate}
Then, using the description of $\xi$ in Tables~\ref{table of local data for p>5} and \ref{table of local data for p=3}, and Table~A.7.\ of \cite{RobertsSchmidt2007}, one can easily check that $\sym^3(\omega_{F,\xi})$ is of type VIII  for case (1), type X for case (2), and it is supercuspidal for case (3).

We can also find the spin $L$-factor $L(s,\sym^3(\pi_p))$ of degree $4$ using the local parameter of $\sym^3(\pi_p)$. In most cases, $L(s,\sym^3(\pi_p))=1$. If $E/\Q_p$ has potentially good reduction at $p\ge 3$, then $L(s,\sym^3(\pi_p))$ is nontrivial precisely when $v_p(\Delta)\equiv 0 \text{ mod } 4$.
We omit the column for $L(s,\sym^3(\pi_3))$ in Table~\ref{table of local data for p=3} since it is similar to Table~\ref{table of local data for p>5}.

Also, one can easily look up $\varepsilon\left(\frac 12,\sym^3(\pi_p)\right)$ for $\pi_p=\chi \times \chi^{-1}$ and $(\gamma(E/\Q_p),\cdot)\text{St}_{\GL(2)}$ using Table~A.9.\ in \cite{RobertsSchmidt2007}. (The additive character implicit in $\varepsilon\left(\frac 12,\sym^3(\pi_p)\right)$ is understood to have conductor exponent $0$.) Finding the $\varepsilon$-factor $\varepsilon\left(\frac 12,\sym^3(\pi_p)\right)$ for a dihedral supercuspidal representation $\pi_p=\omega_{F,\xi}$ requires more work. In this case, we have 
 \begin{equation}
\begin{split}
\varepsilon\left(\frac 12,\sym^3(\omega_{F,\xi})\right)
&=\varepsilon\left(\frac 12, \omega_{F,\xi^3},\psi\right)\varepsilon\left(\frac 12, \omega_{F,\xi},\psi\right)\\
&= \varepsilon\left(\frac 12,\chi_{F/\Q_p},\psi\right)^2\varepsilon\left(\frac 12,\xi^3,\psi\circ\text{tr}\right) \varepsilon\left(\frac 12,\xi,\psi\circ\text{tr}\right).
\end{split}
\end{equation}
Here, $\psi$ is a non-trivial character of $\Q_p$ with the conductor (exponent) $a(\psi)=0$ and tr is the trace map from $F$ to $\Q_p$. When $F/\Q_p$ is unramified, we compute $\varepsilon\left(\frac 12,\xi,\psi\circ\text{tr}\right)$ and $\varepsilon\left(\frac 12,\xi^3,\psi\circ\text{tr}\right)$ for $\pi_p=\omega_{F,\xi}$ using Theorem~3 from \cite{FrohlichQueyrut1973}, and we get 
 \begin{equation}
\varepsilon\left(\frac 12,\sym^3(\omega_{F,\xi})\right)=(-1)^{a(\xi)+a(\xi^3)}.
\end{equation}
This gives us $\varepsilon\left(\frac 12,\sym^3(\pi_p)\right)$ in Table~\ref{table of local data for p>5} and Table~\ref{table of local data for p=3} when $\pi_p=\omega_{F,\xi}$ except for the cases when $E/\Q_3$ satisfies ${\rm S}_6^{'}$ and ${\rm S}_6^{''}$. If $E/\Q_3$ satisfies ${\rm S}_6^{'}$ or ${\rm S}_6^{''}$, using a similar argument as in Lemma~\ref{structure of 1+pF4}, we get the following description of the quadratic extension
 \begin{equation}
 \label{ramified quad field}
 F=\begin{cases}
\Q_3\left(\sqrt{\left(\frac {\Delta'}{3}\right)3}\right) &\text{ if } E \text{ satisfies } {\rm S}_6^{'},\\
 \Q_3(\sqrt{-3}) &\text{ if } E \text{ satisfies } {\rm S}_6^{''},
 \end{cases}
 \end{equation}
where $\Delta'=3^{-v_3(\Delta)}\Delta$. Using \eqref{ramified quad field} and the integration formula for the $\varepsilon$-factors as in $(\epsilon3)$ of \S11 of \cite{Rohrlich1994} we show that 
 \begin{equation}
 \varepsilon\left(\frac 12,\sym^3(\omega_{F,\xi})\right)=\varepsilon\left(\frac 12, \omega_{F,\xi},\psi\right) \text{ if $E/\Q_3$ satisfies  } {\rm S}_6^{'} \text{ or } {\rm S}_6^{''},
\end{equation}
for each possible character $\xi$ of $F/\Q_3$. 
\begin{table}
\caption{The representation type, $\varepsilon$-factor, spin $L$-factor of $\sym^3(\pi_p)$ attached to $E/\Q_p$ with $p\ge 5$. Here $E$ has either potentially multiplicative reduction or potentially good reduction and $\gamma=\gamma(E/\Q_p)$ is the $\gamma$-invariant of $E/\Q_p$.}
	\label{table of local data for p>5}
	\renewcommand{\arraystretch}{2}
	\renewcommand{\arraycolsep}{0.04cm}
	\[ \begin{array}{ cccccc} 
	\toprule
\text{Condition} &\GL(2,\Q_p)& \text{Condition} &\text{Rep.}&\varepsilon\left(\frac 12,\sym^3(\pi_p)\right)& L(s,\sym^3(\pi_p)) \\[-2.4ex]
\text{on } &\text{ rep.}&\text{on} & \text{type of}&&\\[-2.2ex]
E/\Q_p& \pi_p &\pi_p&\sym^3(\pi_p)&&\\
	\toprule
 j(E)\not\in\Z_p&(\gamma,\cdot)\text{St}_{\GL(2)}&(\gamma,\cdot)&\text{IVa}&1&1\\  [-0.14in]
&&\text{ is ram.}&&&\\ 
	\cmidrule{3-6}
	&&(\gamma,\cdot) \text{ is}&\text{IVa}& -1&  \frac{1}{1-p^{-3/2-s}}\\ [-0.14in]
	&&\text{trivial}&&&\\ 
	\cmidrule{3-6}
	&&(\gamma,\cdot)\text{ is unr.}&\text{IVa}&1& \frac{1}{1+p^{-3/2-s}}\\ [-0.14in]
	&&\text{nontrivial}&&&\\ 
	\toprule
  \multirow{2}{*} {\begin{minipage}{19ex}
	\vspace*{-0.06 in}
	\begin{center}
	$j(E) \in \Z_p $
	$\scriptstyle{(p-1)v_p(\Delta) \equiv 0\ \text{mod } 12}$
	\end{center}
	\end{minipage}}&  \multirow{2}{*} { \begin{minipage}{9ex}
	\vspace*{-0.09 in}
	\begin{center}
	$\chi \times \chi^{-1}$
	$a(\chi)=1$
	\end{center}
	\end{minipage}}&\chi|_{\Z_p^{\times}}^2=1&\text{I}&1&1\\ 
 	\cmidrule{3-6}
&&\chi|_{\Z_p^{\times}}^3=1&\text{I}&1&\scriptstyle{ L(s,\chi^3)L(s,\chi^{-3})}\\  
	\cmidrule{3-6}
	&&\chi|_{\Z_p^{\times}}^4=1&\text{I}&1&1\\  
	\cmidrule{3-6}
	&&\chi|_{\Z_p^{\times}}^6=1&\text{I}&1&1\\  
	\toprule
 \multirow{3}{*} { \begin{minipage}{19ex}
	\vspace*{-0.27 in}
	\begin{center}
	$j(E) \in \Z_p $
	$\scriptstyle{(p-1)v_p(\Delta) \not\equiv 0\ \text{mod } 12}$
	\end{center}
	\end{minipage}}& \multirow{3}{*} { \begin{minipage}{16ex}
	\vspace*{0.08in}
	\begin{center}
	$\pi= \omega_{F,\xi}$
	$\scriptstyle{F/\Q_p\text{ is unr.}}$
	$\scriptstyle{a(\xi)=1}$\\
		$\scriptstyle{\xi(\varpi_F)=1}$
	\end{center}
	\end{minipage}}&\xi|_{\OF_F^{\times}}^3=1&\text{X}&-1&\frac{1}{1+p^{-2s}}\\  
		\cmidrule{3-6}
&&\xi|_{\OF_F^{\times}}^4=1&\text{VIII}&1&1\\  
		\cmidrule{3-6}
	&&\xi|_{\OF_F^{\times}}^6=1&\text{X}&1&1\\  
	\bottomrule
	\end{array}\]
\end{table}

\begin{table}
\caption{The Langlands parameter, representation type, and $\varepsilon$-factor of $\sym^3(\pi_3)$ attached to $E/\Q_3$. Here we denote $\ind_{W(\overline{\Q}_3/F)}^{W(\overline{\Q}_3/\Q_3)}(\xi)$ as $\ind_{F}^{\Q_3} \xi$. Also, $\mu$ is an irreducible parameter of $\GL(2,\Q_3)$ with $\text{det}(\mu)=\chi_{F/\Q_3}$, and $\sigma$ is a character of $\Q_3^{\times}$ such that $\xi^3=\sigma\circ N_{F/\Q_3}$. The additive character implicit in the $\varepsilon$-factors is understood to have conductor exponent $0$.}
\label{table of local data for p=3}
\renewcommand{\arraystretch}{1.2}
\renewcommand{\arraycolsep}{0.19cm}
	\[ \begin{array}{ ccccccc } 
	\toprule
	\text{Condition} & \GL(2,\Q_3) &\text{Condition}&\text{Langlands}& \text{Rep.}&\varepsilon\left(\frac 12,\sym^3(\pi_3)\right)\\
	\text{from}&\text{rep.}&\text{on }&\text{parameter}& \text{type of}&\\
	\text{table }\ref{IndexNeron}&\pi_3&\pi_3&\text{for }\sym^3(\pi_3)&\sym^3(\pi_3)&\\
	\toprule
	{\rm P_2} &\chi \times \chi^{-1} & a(\chi)=1  & \chi^3\oplus \chi\oplus \chi^{-1}\oplus\chi^{-3} &{\rm I}& 1\\
	&&\chi|_{\Z_3^{\times}}^2=1&&\\
	\midrule
	{\rm P_3 }& \chi \times \chi^{-1}&a(\chi)=2& \chi^3\oplus \chi\oplus \chi^{-1}\oplus\chi^{-3} &{\rm I}&1\\
	&&\chi|_{\Z_3^{\times}}^3=1&&\\
	\midrule 
	{\rm P_6}&\chi \times \chi^{-1} &a(\chi)=2& \chi^3\oplus \chi\oplus \chi^{-1}\oplus\chi^{-3} &{\rm I}&1\\
	&&\chi|_{\Z_3^{\times}}^6=1&&\\
	\midrule	
	{\rm S_4}&\omega_{F,\xi}&a(\xi)=1&\varphi_{\pi_3}\oplus\varphi_{\pi_3}&{\rm VIII}&1\\
	&{\scriptstyle F={\Q_3}(i)}&\xi(\varpi_F)=-1& \varphi_{\pi_3}=\ind_{F}^{\Q_3} \xi&\\
	&&\xi|_{\OF_F^\times}^4=1&&\\
	\midrule
	{\rm S_3}&\omega_{F,\xi}& a(\xi)=2&\sigma \chi_{F/\Q_3}\oplus\sigma\mu\oplus\sigma &{\rm X}&1\\
	&{\scriptstyle F={\Q_3}(\sqrt{\Delta})}&\xi(\varpi_F)=-1&\sigma\mu=\ind_{F}^{\Q_3} \xi&\\
	&&\xi|_{\OF_{F}^\times}^3=1&&\\
	\midrule
	{\rm S_6}&\omega_{F,\xi}&a(\xi)=2&\sigma \chi_{F/{\Q_3}}\oplus \sigma\mu\oplus\sigma &{\rm X}&-1\\
	&{\scriptstyle F={\Q_3}(\sqrt{\Delta})}&\xi(\varpi_F)=-1&\sigma\mu=\ind_{F}^{\Q_3} \xi&\\
	&&\xi|_{\OF_F^\times}^6=1&&\\
	\midrule
		{\rm S}_6^{'}&\omega_{F,\xi}&a(\xi)=2&\ind_F^{\Q_3} \xi \oplus \ind_F^{\Q_3} \xi^3&\text{super-}&\varepsilon\left(\frac 12,\pi_3\right)\\
	&{\scriptstyle F={\Q_3}(\sqrt{\Delta})}&\xi|_{\OF_F^\times}^6=1&&\text{cuspidal}&\\
	\midrule
	{\rm S}_6^{''}&\omega_{F,\xi}&a(\xi)=4&\ind_F^{\Q_3} \xi \oplus \ind_F^{\Q_3} \xi^3&\text{super-}&\varepsilon\left(\frac 12,\pi_3\right)\\
	&{\scriptstyle F={\Q_3}(\sqrt{\Delta})}&\xi|_{\OF_F^\times}^6=1&&\text{cuspidal}&\\
	\bottomrule
	\end{array}\]
\end{table}
\subsection{Global results}
\label{Section global results}
In this section, we will discuss the global results on the $\sym^3$ lifting of  non-CM elliptic curves over $\Q$. These results are special and more precise versions of Theorem~\ref{RS} and Corollary~\ref{RS-classical} for non-CM elliptic curves.
\begin{theorem}
\label{global theorem}
Let $\pi=\bigotimes_{p}\pi_p$ be the cuspidal automorphic representation of $\GL(2,\A_{\Q})$ with trivial central character attached to a non-CM elliptic curve $E/\Q$ given by the global minimal Weierstrass equation \eqref{W.E of EC} with coefficients in $\Z$. So, the conductor $a(\pi)$ of $\pi$ equals the conductor $N$ of $E$. Suppose that $E$ has good or potentially multiplicative reduction at $p=2$. Then there exists a cuspidal automorphic representation $\Pi=\bigotimes_{p}\Pi_p$ of $\GSp(4,\A_{\Q})$ with trivial central character which is unramified at each prime $p$ not dividing $N$, such that
	\begin{enumerate}
		\item[(i)] $\Pi_p$ is a generic representation at each finite prime $p$. 
		\item[(ii)] $\Pi_{\infty}$ is a holomorphic discrete series representation with  minimal $K$-type $(3,3)$.
		\item[(iii)] $L(s,\Pi)=L(s, \pi, \sym^3)$.
	\end{enumerate}
Moreover, the conductor of $E$ and the conductor $a(\Pi)$ of $\Pi$ are as follows
\begin{equation}
\label{N and M}
N= \prod_{p \mid \Delta}p^i \quad\text{ and }	\quad a(\Pi)= \prod_{p \mid \Delta}p^k,
\end{equation}
where the values of $i, k$ are given in Table~\ref{table for i,k} for each $p$ dividing the discriminant $\Delta$ of \eqref{W.E of EC}.
\end{theorem}
\begin{proof}
Most of the statements of the theorem follow from Theorem~\ref{RS}. We only need to show that the conductor $N$ of $E$ and the conductor $a(\Pi)$ of $\Pi=\sym^3(\pi)$ are given by \eqref{N and M}. Since $N=a(\pi)=\otimes_p p^{a(\pi_p)}$ and $a(\Pi)=\otimes_p p^{a(\Pi_p)}$, considering all possible values of $a(\pi_p)$ and $a(\Pi_p)=a(\sym^3(\pi_p))$ from Tables~\ref{conductor of sym^3 for potentially multiplicative}, \ref{conductor of sym^3 for p>3}, and \ref{conductor of sym^3 for p=3}, we see that $N$ and $a(\Pi)$ are given by \eqref{N and M} where $i,k$ are given in Table~\ref{table for i,k}. This concludes the proof.
\end{proof}
\begin{table}
\caption {Exponents $i, k$ of $p$ in $N=a(\pi)$ and $a(\sym^3(\pi))$ respectively. We use this table in Theorem~\ref{global theorem} and Corollary~\ref{global corollary1}. Here, $e=\frac{12}{\left(v_p(\Delta),12\right)}$ with $v_p$ being the $p$-adic valuation.}
\label{table for i,k}
\renewcommand{\arraystretch}{1}
\renewcommand{\arraycolsep}{0.45 cm}
\[ \begin{array}{ ccccc } 
\toprule
\text{Condition} &\text{Condition}&i & k& \text{Reduction type}\\
\text{on }  p\mid N &\text{on } E&&&\text{of } E \text{ at } p\\
\toprule
\text{for any } p&v_p(\Delta) > 0, &1&3&\text{multiplicative}\\
p \mid\mid N&v_p(c_4)=0 &&&\text{reduction}\\
\midrule
\text{for any } p&j(E) \not\in \Z_p&i=2 \qquad \text{ if }p\ge 3&2i&\text{additive,}\\
p^i \mid\mid N&&i\in \left\{4,6\right\} \text{ if } p=2&&\text{potentially}\\
i \ge 2 &&&&\text{multiplicative}\\
\midrule
\text{for } p \ge 5&j(E)\in \Z_p,&2&4&\text{additive,}\\
p^2 \mid\mid N& e \in \left\{2,4,6\right\}&&&\text{potentially good}\\
\midrule
\text{for } p \ge 5&j(E)\in \Z_p,&2&2&\text{additive,}\\
p^2 \mid\mid N&e=3&&&\text{potentially good}\\
\midrule
\text{for } p=3&j(E)\in \Z_3,&2&4&\text{additive,}\\
p^2 \mid\mid N& E \text{ satisfies} &&&\text{potentially}\\
&{\rm P_2\ or\ S_4}&&&\text{good}\\
&\text{from Table } \ref{IndexNeron}&&&\\
\midrule
\text{for } p=3&j(E)\in \Z_3,&3&5&\text{additive,}\\
p^3 \mid\mid N& E \text{ satisfies } {\rm S}^{'}_6&&&\text{potentially}\\
&\text{from Table } \ref{IndexNeron}&&&\text{good}\\
\midrule
\text{for } p=3&j(E)\in \Z_3, &4&4&\text{additive,}\\
p^4 \mid\mid N& E \text{ satisfies} &&&\text{potentially}\\
&{\rm P_3\ or\ S_3}&&&\text{good}\\
&\text{from Table } \ref{IndexNeron}&&&\\
\midrule
\text{for } p=3&j(E)\in \Z_3,&4&6 &\text{additive,}\\
p^4 \mid\mid N& E \text{ satisfies} &&&\text{potentially}\\
&{\rm P_6\ or\ S_6}&&&\text{good}\\
&\text{from Table } \ref{IndexNeron}&&&\\
\midrule
\text{for } p=3&j(E)\in \Z_3, &5&7&\text{additive,}\\
p^5 \mid\mid N& E \text{ satisfies } {\rm S}^{''}_6&&&\text{potentially}\\
&\text{from Table } \ref{IndexNeron}&&&\text{good}\\
\bottomrule
\end{array}\]
\end{table}
\noindent Since there is a natural correspondence between the cuspidal automorphic representations of $\GSp(4,\A_{\Q})$ and paramodular forms as discussed in Corollary~\ref{RS-classical}, we get the following corollary as an immediate consequence of Theorem~\ref{global theorem}.
\begin{corollary}
	\label{global corollary1}
	Let $E$ be a non-CM elliptic curve over $\Q$ given by the global minimal Weierstrass equation \eqref{W.E of EC} with coefficients in $\Z$. Suppose that $E$ has good or potentially multiplicative reduction at $p=2$. Then there is a cuspidal paramodular newform $f$ of degree $2$, weight $3$, and level $M$ such that $ L(s,f)=L(s, E, \sym^3)$.
	Moreover, the conductor $N$ of $E$ and the level $M$ of $f$ are given by
	\begin{equation*}
	N= \prod_{p \mid \Delta}p^i \quad\text{ and }	\quad M= \prod_{p \mid \Delta}p^k,
	\end{equation*}
where the values of $i,k$ are given in the Table \ref{table for i,k} for $p$ dividing the discriminant $\Delta$ of \eqref{W.E of EC}.
\end{corollary}
\noindent Furthermore, if the conductor of an non-CM elliptic curve $E/\Q$ is given along with the Weierstrass equation, then we have the following refinement of Theorem~\ref{global theorem}.
\begin{corollary}
	\label{global corollary2}
	Let $\pi=\bigotimes_{p}\pi_p$ be the cuspidal automorphic representation of $\GL(2,\A)$ attached to a non-CM elliptic curve $E/\Q$ given by the global minimal Weierstrass equation \eqref{W.E of EC} with coefficients in $\Z$ and the conductor $N$. Suppose that $E$ has good or multiplicative reduction at $p=2$. Then $\Pi=\sym^3(\pi)$ is a cuspidal automorphic representation of $\GSp(4,\A_{\Q})$ with trivial central character, which is unramified at each prime $p$ not dividing $N$ and satisfies the properties (i), (ii), (iii) as in Theorem~\ref{global theorem}. Moreover, the conductor $a\left(\sym^3(\pi)\right)$ of $\sym^3(\pi)$ is given by
	\begin{equation}
	\label{formula of sym3 conductor}
	a\left(\sym^3(\pi)\right)= N\prod_{\substack{p \mid N \\ v_p(\Delta) \not\equiv 0 \ \mathrm{mod}\ 4}}p^2.
	\end{equation}
\end{corollary}
\begin{proof}
	Here, we only need to prove the formula \eqref{formula of sym3 conductor} for $a\left(\sym^3(\pi)\right)$. Note that  \eqref{formula of sym3 conductor} is a consequence of the following facts from Table~\ref{table for i,k}:
	\begin{equation}
	\label{relation between i,k}
	\begin{split}
	k=i  &\quad\text{ if and only if }\quad v_p(\Delta) \equiv 0 \text{ mod }4. \\
	 k=i+2  &\quad\text{ if and only if }\quad v_p(\Delta) \not\equiv 0 \text{ mod }4.
	\end{split}
	\end{equation}
	To see this, note that $k=i$ in Table~\ref{table for i,k} precisely in the following cases:
	\begin{enumerate}
		\item $j(E)\in \Z_p$ and $e=3$ for $p\ge 5$
		\item $j(E)\in \Z_3$ and  $E$ satisfies ${\rm P_3\ or\ S_3}$ from Table~\ref{IndexNeron}.
	\end{enumerate}
	Now, $e=\frac{12}{\left(v_p(\Delta),12\right)}=3$ if and only if $v_p(\Delta) \equiv 0 \text{ mod }4$. Also, from Table~\ref{IndexNeron}, we see that $E$ satisfies ${\rm P_3\ or\ S_3}$ if and only if $v_p(\Delta) \equiv 0 \text{ mod }4$. Hence, $k=i$ if and only if $v_p(\Delta) \equiv 0 \text{ mod }4$. Moreover, observe that $k=i+2$ in Table~\ref{table for i,k} for all other cases except for the case when $E$ has additive but potentially multiplicative reduction at $p=2$. Since we have omitted this case in the statement of the corollary, we get $k=i+2$ if and only if $v_p(\Delta) \not\equiv 0 \text{ mod }4$. This proves \eqref{relation between i,k} and hence proves  \eqref{formula of sym3 conductor}.
\end{proof}
\begin{proof}[\textbf{Proof of Theorem~\ref{theorem in intro}}]
Most parts of Theorem~\ref{theorem in intro} including the statement (i) follow from Corollary~\ref{global corollary2}. Let $\pi=\otimes \pi_p$ be the cuspidal automorphic representation of $\GL(2,\A_{\Q})$ associated to $E/\Q$. The assertion (ii) follows from Table~\ref{table of local data for p>5} and \ref{table of local data for p=3}, and the fact that $L_p(s,f)=L(s,\sym^3(\pi_p))$. One can easily see the assertion (iii) using Table~\ref{table of local data for p>5} and \ref{table of local data for p=3}, and the fact that the Atkin-Lehner eigenvalue of $f$ at a finite place $p$ is equal to $\varepsilon\left(\frac 12,\sym^3(\pi_p)\right)$. Also, observe that $w(E/\Q_3)=\varepsilon(\frac12,\pi_3)$, where the local root number $w(E/\Q_3)$ can be computed from \cite{Kobayashi2002}.
\end{proof}

\vspace{0.1in}
\textbf{Acknowledgements.} I would like to thank Ralf Schmidt for guiding me patiently throughout this work. 
I am grateful to Cris Poor and David Yuen for helpful discussions. I would also like to thank the referee for the detailed comments and suggestions.
\bibliographystyle{acm}
\bibliography{Paramodular_EC_ref.bib}
\vspace{5ex}
\noindent Department of Mathematics, Fordham University, Bronx, New York 10458, USA.

\noindent E-mail address: {\tt manami.roy.90@gmail.com}
\end{document}